\documentclass[10pt]{amsart}
\usepackage{amssymb,latexsym,pstricks}
\usepackage[normalem]{ulem}
\usepackage{orcidlink}
\usepackage{cancel}
\usepackage{changes}
\topskip  \usepackage{color}
\usepackage{hyperref}
\usepackage{tensor}
\usepackage{enumitem}
\usepackage{mathabx}
\usepackage{comment}
\usepackage{amsfonts,amssymb,amsmath,amsthm,graphicx,latexsym}
\usepackage{mathrsfs,dsfont}
\usepackage{pstricks,pst-node,pst-plot}
\usepackage{pgfplots}
\usepackage{tikz}
\usepackage{bm}
\usepackage{caption}
\theoremstyle{thmstyleone}%
\newtheorem{theorem}{Theorem}[section]
\newtheorem{proposition}[theorem]{Proposition}%
\newtheorem{corollary}[theorem]{Corollary}

\theoremstyle{thmstyletwo}%
\newtheorem{example}[theorem]{Example}%
\newtheorem{remark}[theorem]{Remark}%

\theoremstyle{thmstylethree}%
\newtheorem{definition}[theorem]{Definition}%

\newcommand{\tp}{\mathbin{\hbox{$\bigcirc$\hbox to 0pt{\hspace{-0.81em}$\scriptstyle\top$\hfil}}}}

\begin{document}
\title[Normal ordering in the $(p,q)$-deformed generalized Weyl algebra]{Normal ordering in the $(p,q)$-deformed generalized Weyl algebra. II: Interpretation in terms of rook placements}
\author[T. Mansour]{Toufik Mansour $^{1, \orcidlink{0000-0001-8028-2391}}$}
\address{$^1$ Department of Mathematics, University of Haifa, 3498838 Haifa, Israel, tmansour@univ.haifa.ac.il}
\author[L. Oussi]{Lahcen Oussi$^{2, \orcidlink{0000-0001-9804-0761}}$}
\address{$^2$ Faculty of Pure and Applied Mathematics, Wrocław University of Science and Technology, Wybrzeże Wyspiańskiego 27, 50-370 Wrocław, Poland, lahcen.oussi@pwr.edu.pl, oussimaths@gmail.com}
\author[M. Schork]{Matthias Schork $^{3, \orcidlink{0000-0002-9759-6618}}$}
\address{$^3$ Institute for Mathematics, Würzburg University, Emil-Fischer Str. 40, 97074 Würzburg, Germany, matthias.schork@mathematik.uni-wuerzburg.de}


\begin{abstract}
In this paper, we investigate the combinatorial structure arising from the $(p, q)$-deformed generalized Weyl algebra generated by variables $X, Y$, and $Z_p$, satisfying the $(p, q)$-commutation relations $XY-qYX=h Y^sZ_{p}, XZ_p=pZ_pX$, and $Z_pY=pYZ_p$, where $s\in \mathbb{N}_0$. Our primary objective is to use the normal ordering process defined by these relations to develop a novel model of $(p, q)$-deformed rook theory. Specifically, we introduce a new framework of $(p, q)$-deformed $s$-rook numbers derived from this normal ordering process. Utilizing these combinatorial models, we provide explicit combinatorial interpretations for the associated $(p, q)$-generalized Stirling numbers via rook placements on staircase boards. Our results extend several classical and recent formulations in the literature to the general $p\neq 1$ setting.
\end{abstract}
\maketitle
\section{Introduction}\label{sec1}
A commutation relation characterizes the behavior of given operators $X$ and $Y$ when their order of operation is reversed. To formalize this, we use the standard commutator  $[X, Y]:=XY-YX$. If the operators $X$ and $Y$ commute, the commutator vanishes. Noncommutative structures are widespread in both mathematics and physics, appearing naturally in matrix theory, quaternions, and Lie algebras, etc. These algebraic frameworks are essential in discrete mathematics and quantum theory. For instance, analyzing the quantum harmonic oscillator requires systematic reordering rules to evaluate mixed expressions. The primary mathematical setting for studying such relations is the standard {\em Weyl algebra}, which is the complex unital algebra generated by two operators $X$ and $Y$ satisfying the commutation relation 
\begin{equation}\label{Walg}
			[X, Y]=I, 
\end{equation} 
where $I$ denotes the identity operator. From a functional perspective, these generators are often represented as operators acting on smooth functions of a single variable, where $X$ acts as the derivative operator $D=\frac{d}{dx}$ and $Y$ acts as the multiplication operator $M$, defined by $(Mf)(x)=xf(x)$. Therefore, this concrete realization makes the Weyl algebra the classical example of a ring of differential operators. In the context of quantum mechanics, this structure is realized via the annihilation operator $\hat{a}$ and creation operator $\hat{a}^{\dag}$ of the harmonic oscillator.

An expression in the Weyl algebra is said to be in  {\em normal ordered form} if all letters $Y$ stand to the left of all the letters $X$. Finding the normal ordered form of an arbitrary operator word is a fundamental problem with numerous applications. In 1823, Scherk showed that normal ordering the specific powers word $(YX)^n$ naturally involves the Stirling numbers of  the second kind as coefficients: 
\begin{equation}\label{Scherk}
	(YX)^n=\sum_{k=1}^n S(n,k) Y^kX^k.
\end{equation}
For a comprehensive historical overview, see~\cite{TMMS2016} (or also \cite{BF2011}, where Scherk's dissertation is discussed). This connection between operator words and combinatorics was later rediscovered independently within the physics community by Katriel in 1974 using creation and annihilation operators. Since then, numerous generalizations have been investigated, see, e.g.,~\cite{BF2011,TMMS2016,MS2021}.

In the past years, a great deal of research has been dedicated to $q$-deformed structures by various authors. Generally, a $q$-deformation involves introducing a parameter $q$ into the defining algebraic relations such that the undeformed structures are recovered in the limiting case as  $q\to 1$. For instance, the {\em $q$-deformed Weyl algebra} is given by $XY-qYX=I$. Another fundamental structure is the {\em quantum plane}~\cite{YIM1988}, whose generating operators satisfy the $q$-commutation relation 
\begin{equation}\label{qcrintro}
	XY=qYX.
\end{equation}
The algebraic properties and consequences of~\eqref{qcrintro} are deeply connected to identities for $q$-special functions (see, e.g., \cite{KHT1995, KTH1997}). Just as in the classical setting, normal ordering in these $q$-deformed frameworks yields formulas analogous to~\eqref{Scherk} where the coefficients become $q$-deformed Stirling numbers of the second kind.  Within the physics literature, many $q$-deformed models have been extensively studied. Parallel to the undeformed setting, the commutation relations of the $q$-deformed Weyl algebra are realized explicitly through $q$-deformed creation and annihilation operators.

More recently, these mathematical models have been generalized to multiparameter $(p,q)$-deformations, introducing a second parameter $p$ such that, for $p\to 1$, the single parameter $q$-deformation is recovered. Driven by both algebraic and physical motivations, $(p, q)$-extensions of the Weyl algebra and associated oscillator have been constructed in~\cite{JG2016,EKLS1993}, leading to $(p, q)$-deformed Stirling numbers upon normal ordering reductions~\cite{KK1992} and motivating the study of generalized $(p, q)$-special functions~\cite{BK1994,Jag1997}. Simultaneously, the $(p, q)$-deformed Stirling numbers were introduced in relation to a $(p, q)$-statistic for rook placements on staircase board~\cite{MWDW1991}. Subsequently, various algebraic properties and further generalizations have been developed.  

Rook theory was introduced by Kaplansky and Riordan \cite{KR1946}, and was developed subsequently by many authors, see, e.g., \cite{KBJR2003, DFMPS1970, AGJR1986, JGJH2000, JGJJDRRW1976, JRMW2004, JR1958, RPS1997}. Rook theory treats the way of placing a collection of rooks on a board, in particular {\it Ferrers boards}. Navon \cite{AMN1973} discovered the connection between normal ordering and rook theory, and Varvak \cite{AV2005} brought this into the precise form used since then. Recall that for a word $w$ in letters $X$ and $Y$ satisfying \eqref{Walg} we say that it is in a normal ordered form if all letters $Y$ stand to the left of all letters $X$. That is, $w=\sum_{i, j} r_{i, j}Y^iX^j$ for some coefficients $r_{i, j}$. The connection between Ferrers boards and normal ordering is as follows: For a word $w$ in $X$ and $Y$, start from $(0, 0)$ in $\mathbb{Z}^2$ and represent the letter $X$ as a step to the right and the letter $Y$ as a step up, reading the word from left to right. Then the resulting path outlines a Ferrers board denoted by $B_w$ and the normal ordering coefficients of $w$ are given by rook numbers of the Ferrers board $B_w$. For instance, $w=(YX)^n$ outlines the {\em staircase board} $J_n$, and it is well known that the rook numbers of the staircase board are given by Stirling numbers of the second kind, in accordance with \eqref{Scherk}. This connection has been studied and generalized by many authors. In particular, Varvak \cite{AV2005} gave an interpretation of the normal ordering coefficients of words in $X$ and $Y$ satisfying $XY=YX+hY^s$ (with $s\in \mathbb{N}_0$) in terms of the $i$-rook numbers introduced by Goldman and Haglund \cite{JGJH2000}. This was extended to the $q$-deformed case by Celeste et al. \cite{CCG2017}.

The main objective of this paper is to establish the  combinatorial behavior of the variables $X, Y,$ and $Z_p$ satisfying the following $(p, q)$-commutation relations: 
\begin{align}
	XY-qYX &= hf(Y)Z_p,\label{pqcr1} \\
	Z_pY&=pYZ_p, \label{pqcr2}\\
	XZ_p&=pZ_pX, \label{pqcr3}\\
	Z_1 &=I,\label{pqcr4}
\end{align}
where $f$ is a polynomial function. 
Within this algebraic framework, a word $\omega$ in the letters $X, Y$, and $Z_p$ is said to be in {\em normal ordered form} if all the letters $Y$ stand to the left of all the letters $Z_p$ which stand to the left of all the letters $X$. That is, $\omega$ can be uniquely expanded as a linear combination of the form $\omega=\sum_{k,l,m} a_{k,l,m}Y^kZ_p^lX^m$ for some coefficients $a_{k,l,m}$, the {\em normal ordering coefficients}. For instance, setting $p=1$ yields the $q$-commutation relation $XY-qYX = hf(Y)$ studied in \cite{CCG2017,TMMS2011}. Additionally, by taking the limit $q\to 1$, the resulting algebra was studied in depth in \cite{BLO2013,BLO2015,BLO2015a}. Conversely,  by identifying  $X$ with the $(p, q)$-derivative operator $D_{p, q}$, $Y$ with the operator of multiplication with $x$, and $Z_p$ with the Fibonacci operator \cite{GCMDCH}, one obtains the $(p, q)$-commutation relation established in \cite{LO2021}.

In this work, we specialize to the case where $f(Y)=Y^s$ is a monomial, i.e., to the {\em $(p,q)$-deformed generalized Weyl algebra}. Under this choice, we provide a combinatorial interpretation of the normal ordering coefficients in terms of rook theory. If $\omega$ is a word in the letters $X,Y$ and $Z_p$ in the $(p,q)$-deformed generalized Weyl algebra, we introduce as intermediate step so called {\em augmented Ferrers boards} where the letters $X$ and $Y$ of $\omega$ outline a Ferrers board in the usual fashion but where in addition the letters $Z_p$ are represented by particular labels on this Ferrers board. This is then translated into appropriate $(p,q)$-weights for $s$-rook placements, and, thus, to adapted $(p,q)$-deformed $s$-rook numbers. 

This paper is the second part of a series of 3 papers. In a prequel \cite{TLM1}, we studied normal ordering words in the $(p,q)$-deformed generalized Weyl algebra from a more algebraic perspective. In further work \cite{TLM3}, the binomial formula for $(X+Y)^n$ is studied in the $(p,q)$-deformed generalized Weyl algebra.

The remainder of this paper is organized as follows. In Section~\ref{sprelim}, we provide a brief overview of the foundational normal ordering results that will be utilized in the subsequent sections. Section~\ref{srook} is dedicated to establishing the connection between normal ordering and rook theory. In Section~\ref{spqrook}, we  introduce the $(p, q)$-deformed rook numbers adapted to this normal ordering framework. Here we derive several operational identities, extending the previous results to the case $p\neq 1$ and providing a combinatorial interpretation for the $(p, q)$-deformed generalized Stirling numbers. Finally, in Section~\ref{concl}, we present some conclusions and prospective outlook.

\section{Preliminaries}\label{sprelim}
In this section, we briefly recall some basic definitions and facts about the generalized Stirling numbers associated with the $(p, q)$-deformed Weyl algebra to fix the notation and provide the necessary background for the present paper, see \cite{TLM1} for more details.

Let $p$ and $q$ be two indeterminates. For any non-negative integer $n\in\mathbb{N}_0:=\mathbb{N}\cup\{0\}$, we define the {\em twin-basic numbers} by
$$
[n]_{p, q}=\frac{p^n-q^n}{p-q},
$$ 
and are viewed as a $(p,q)$-deformation of the natural numbers. Setting $p=1$, one recovers the conventional $q$-deformed numbers, namely $[n]_{1, q}=\frac{1-q^n}{1-q}=[n]_{q}$. Alternatively, one has $[n]_{p, q}=p^{n-1}+p^{n-2}q+\cdots + pq^{n-2}+q^{n-1}$, which reduces, for $p=1$,  to the well-known geometric series $[n]_{q}=1+q+q^2+\cdots + q^{n-1}$. The associated $(p,q)$-factorials and $(p, q)$-binomial coefficients are given, respectively, by
$$
[n]_{p, q}!=[n]_{p, q}[n-1]_{p, q}\cdots [2]_{p, q}[1]_{p, q},\,\,\, \text{with }\, [0]_{p, q}!=1,
$$
and 
$${n\choose k}_{p, q}=\frac{[n]_{p, q}!}{[n-k]_{p, q}![k]_{p, q}!}.$$ 
Furthermore, the corresponding {\em $(p,q)$-derivative} operator $D_{p,q}$ acts on a given function $f(x)$ as
$$
(D_{p,q}f)(x):=\frac{f(px)-f(qx)}{(p-q)x}.
$$
In the limiting case where  $p=1$, this operator reduces to the {\em Jackson-derivative} $D_q$, which is explicitly given by $(D_{1,q}f)(x)=\frac{f(x)-f(qx)}{(1-q)x}=(D_{q}f)(x)$.  The action of $D_{p,q}$ on monomials is given, for $n \in \mathbb{N}$, by 
$$
D_{p,q}x^n=[n]_{p, q}x^{n-1}.
$$
We refer the reader to~\cite{LO2024} for more detailed discussion on $D_{p,q}$ and related references.

Let us recall the following definition from~\cite{TLM1}. 
\begin{definition}\label{DefpqGenWeyl}
	The {\em $(p, q)$-deformed generalized Weyl algebra $A_{s;h|p,q}$} is defined, for $s\in\mathbb{N}_{0}, h\in\mathbb{C}\setminus\{0\}$ and $p, q\in\mathbb{C}$, as the complex unital algebra generated by variables $X,Y$ and $Z_p$ satisfying 
	\begin{equation}\label{pqdefgenWeyl}
		XY-qYX = h Y^sZ_p, \,\, Z_pY=pYZ_p, \,\, XZ_p=pZ_pX, \,\, Z_1 =I.
	\end{equation}
\end{definition}
It is worth mentioning that, for $p=1$, the above definition reduces to the $q$-deformed generalized Weyl algebra $A_{s;h|1,q}=A_{s;h|q}$ \cite[Definition 1.25]{TMMS2016}, which  was first introduced in \cite{TMMS2011}. The special case where $p=q=h=1$ was established by Burde \cite{Burde2005} (in the context of matrices $X$ and $Y$), and Varvak \cite{AV2005} (see also \cite[Section 9.1]{TMMS2016}). 

 There is a deep connection between normal ordering and Stirling numbers. Namely, in this context,  the {\it generalized Stirling numbers} $\mathfrak{S}_{s;h}(n,k|p,q)$ were introduced as normal ordering coefficients of $(YX)^n$ in $A_{s;h|p,q}$ as follows (see \cite[Definition 3.3]{LO2021}),
  \begin{equation}\label{StirGen}
  	(YX)^n=\sum_{k=0}^n \mathfrak{S}_{s;h}(n,k|p,q) \, Y^{s(n-k)+k}Z_p^{n-k} X^{k}.
  \end{equation}
  Moreover, they satisfy the recurrence relation \cite[Theorem 3.5]{LO2021}
  \begin{equation}\label{Stirlingrecu}
  	\mathfrak{S}_{s;h}(n+1,k|p,q)=p^{n-k+1}q^{s(n-k+1)+k-1}\mathfrak{S}_{s;h}(n,k-1|p,q)+h[s(n-k)+k]_{p, q}\mathfrak{S}_{s;h}(n,k|p,q).
  \end{equation}
  The generalized Stirling numbers $\mathfrak{S}_{s;h}(n,k|p,q)$ in~\eqref{StirGen} are defined as normal ordering coefficients of the particular word $(YX)^n$. For arbitrary words $\mathrm{w}_{{\bf m}, {\bf n}}=Y^{n_r}X^{m_r}\cdots Y^{n_1}X^{m_1}$, with ${\bf m}=(m_1, \ldots, m_r)\in\mathbb{N}_0^{r}$ and ${\bf n}=(n_1, \ldots, n_r)\in\mathbb{N}_0^r$,  the following generalized Stirling numbers are considered in~\cite{TLM1}:
  \begin{definition}\label{rooknum}
  	Let $X, Y$ and $Z_p$ be variables satisfying the commutation relations \eqref{pqdefgenWeyl} of the $(p, q)$-deformed generalized Weyl algebra $A_{s;h|p,q}$. The $(p, q)$-deformed generalized Stirling numbers $S_{s, h; p, q}^{{\bf m}, {\bf n}}[k]$ are defined as normal ordering coefficients of the string $\mathrm{w}_{{\bf m}, {\bf n}}=Y^{n_r}X^{m_r}\cdots Y^{n_1}X^{m_1}$, i.e., by
  	\begin{equation}\label{pqrook}
  		\mathrm{w}_{{\bf m}, {\bf n}}=\sum_{k=m_1}^{|{\bf m}|}S_{s, h; p, q}^{{\bf m}, {\bf n}}[k]Y^{|{\bf n}|-(|{\bf m}|-k)(1-s)}Z_{p}^{|{\bf m}|-k}X^{k},
  	\end{equation}
  \end{definition}
  where $|{\bf m}|=\sum_{j=1}^{r}m_j$ and $|{\bf n}|=\sum_{j=1}^{r}n_j$.
  Note that, for the special case ${\bf m}={\bf 1}_r=(
  \raisebox{0pt}[\height][0pt]{$\underbrace{1, 1\ldots1}_{r\ \text{times}})$}$ and ${\bf n}={\bf 1}_r=(
  \raisebox{0pt}[\height][0pt]{$\underbrace{1, 1\ldots1}_{r\ \text{times}})$}$, we obtain
  \vspace{.5cm}
  \begin{equation}
  	(YX)^r=\sum_{k=1}^{r}S_{s, h; p, q}^{{\bf 1}_r,{\bf 1}_r}[k] Y^{s(r-k)+k}Z_{p}^{r-k}X^{k},
  \end{equation}
  which by comparison with \eqref{StirGen} shows that $S_{s, h; p, q}^{{\bf 1}_r,{\bf 1}_r}[k]=\mathfrak{S}_{s;h}(r,k|p,q)$, for $r \geq 1$.

\section{Connections Between Normal Ordering and Rook Placements}\label{srook}
In this section, we briefly recall the connection between normal ordering and rook numbers. Before we do this, we first recall the definitions and some notions of rook theory. 

\subsection{File numbers, rook numbers and $i$-rook numbers}\label{SectRook} 
Let us first recall some basic tools about partitions from~\cite{TMMS2011}. A {\em partition} $\lambda=(\lambda_1,\lambda_2,\dots,\lambda_{\ell}) \equiv \lambda_1\lambda_2\ldots\lambda_{\ell}$ is a weakly decreasing sequence of positive integers. We denote the sum of the parts of $\lambda$ by $|\lambda|=\sum_{i=1}^{{\ell}}\lambda_i$. Following~\cite{TLM1}, we associate with each partition $\lambda = \lambda_1\dots \lambda_{\ell}$ a Young diagram $Y_{\lambda}$ (a left-justified diagram consisting of $|\lambda|$ boxes) of shape $\lambda$, where the $j$-th column contains $\lambda_j$ boxes, for $1\leq j \leq \ell$. 

We align the columns on the top so we identify the Young diagram $Y_{\lambda}$ with the Ferrers board denoted by $B_{\lambda}$. Given a Ferrers board $B$, we call a placement of $k$ rooks in $B$ such that there is at most one rook in each column a {\em file placement of $k$ rooks}, or also a $k$-{\em file placement}. The set of all $k$-file placements on $B$ will be denoted by $\mathcal{F}_k(B)$, and $f_k(B)=|\mathcal{F}_k(B)|$ will be called $k$-{\em th file number of $B$}. A $k$-{\em rook placement} is a special kind of file placement where in addition no two rooks are in the same row (i.e., the $k$ rooks are {\it non-attacking}). The set of all $k$-rook placements on $B$ will be denoted by $\mathcal{R}_k(B)$, and $r_k(B)=|\mathcal{R}_k(B)|$ will be called $k$-{\em th rook number of $B$}. See Figure~\ref{FigFile} for an example of a $4$-file placement on the Ferrers board $B_{33211}$ associated to the partition $\lambda=33211$.

\begin{figure}[h]
\centering
\begin{tikzpicture}

\draw (0,0.5) -- (1,0.5);

\draw (0.5,0.5) -- (0.5,1.5);
\draw (1,0.5) -- (1,2);  
\draw (1.5,1) -- (1.5,2); 
\draw (2,1.5) -- (2,2);
\draw (2.5,1.5) -- (2.5,2);  

\draw (0,0.5) -- (0,2);
\draw (0.5,1) -- (0.5,2); 
\draw (0,1) -- (1.5,1); 
\draw (0,1.5) -- (2.5,1.5); 
\draw (0,2) -- (2.5,2);

\filldraw (0.25,0.75) circle (3pt);
\filldraw (0.75,1.25) circle (3pt);
\filldraw (1.25,1.75) circle (3pt);
\filldraw (1.75,1.75) circle (3pt);
\end{tikzpicture}
\caption{The Ferrers board $B_{33211}$ with a $4$-file placement which is not a $4$-rook placement.}\label{FigFile}
\end{figure}
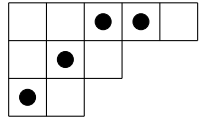
\begin{example}
If the board $B_{\lambda}$ has $\ell$ columns, then clearly $f_k(B_{\lambda})=r_k(B_{\lambda})=0$ if $k > \ell$. Furthermore, $f_{1}(B_{\lambda})=r_1(B_{\lambda})=|\lambda|$, the number of boxes. For $k=\ell$, we have $f_{\ell}(B_{\lambda})=\prod_{j=1}^{\ell}\lambda_j$ since we can choose the boxes in the $\ell$ columns independently. Let us consider the {\em staircase board} $J_n=B_{(n-1)\cdots 321}$. It is well known (see, e.g., \cite[Section 2.4.4]{TMMS2016} and the references given therein) that 
\begin{equation}\label{StairEq}
r_{n-k}(J_n)=S(n,k), \,\,\,\, f_{n-k}(J_n)=|s(n,k)|,
\end{equation}
where $S(n,k)$ (resp., $s(n,k)$) denote the Stirling numbers of the second (resp., first) kind. 
\end{example}

The above model of rook and file placements was generalized by Goldman and Haglund \cite{JGJH2000} as follows. Let $i\in \mathbb{N}_0$ and let $B$ be a Ferrers board. The {\em $i$-row creation rule} for rook placements means that each time we choose to place a rook, we create $i$ new rows to the left of the box. The rooks are placed on the Board $B$ going from right to left. When creating $i$ new rows, we interpret the resulting top row as the row belonging to the rook just placed while the $i$ rows beneath it as new rows.
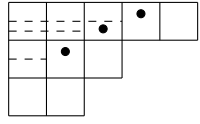
\begin{figure}[h]
\centering
\begin{tikzpicture}

\draw (0,0.5) -- (1,0.5);

\draw (0.5,0.5) -- (0.5,1.5);
\draw (1,0.5) -- (1,2);  
\draw (1.5,1) -- (1.5,2); 
\draw (2,1.5) -- (2,2);
\draw (2.5,1.5) -- (2.5,2);  

\draw[dashed] (0,1.25) -- (0.5,1.25);
\draw[dashed] (0,1.75) -- (1.5,1.75);
\draw[dashed] (0,1.62) -- (1,1.62);

\draw (0,0.5) -- (0,2);
\draw (0.5,1) -- (0.5,2); 
\draw (0,1) -- (1.5,1); 
\draw (0,1.5) -- (2.5,1.5); 
\draw (0,2) -- (2.5,2);

\filldraw (0.75,1.35) circle (1.5pt);
\filldraw (1.25,1.65) circle (1.5pt);
\filldraw (1.75,1.85) circle (1.5pt);
\end{tikzpicture}
\caption{The Ferrers board $B_{33211}$ with a $3$-rook placement with 1-row creation rule.}\label{FigFile2}
\end{figure} 

Following Goldman and Haglund \cite{JGJH2000}, we define the $i$-rook numbers as follows.
\begin{definition} Let $i\in \mathbb{N}_0$. Given a Ferrers board $B$, the {\em $k$-th $i$-rook number} $r_k^{(i)}(B)$ is the number of ways to place $k$ non-attacking rooks on the board $B$ going from right to left, creating $i$ new rows to the left of each rook.  
\end{definition}
Clearly, for $i=0$ one recovers the conventional rook numbers. For $i=1$, each time we place a rook, we create a new row to the left. Thus, for the next rook to place, in each column to the left one has the same amount of boxes to choose from as the column has. In other words, the placement of $k$ rooks with 1-row creation rule corresponds to a $k$-file placement. Thus,  
\begin{equation}\label{RookRed}
r_k^{(0)}(B)=r_k(B), \,\,\,\, r_k^{(1)}(B)=f_k(B). 
\end{equation}

\begin{remark}\label{IRookEq}
In the original definition of the $i$-rook numbers given above, one creates $i$ new rows to the left of the rook placed, but subsequent rooks have to be placed such that they are non-attacking, i.e., from the resulting $(i+1)$ rows to the left of the rook placed only $i$ can be used for the subsequent placements of a rook. Equivalently, we can interprete this rule as creating $(i-1)$ new rows to the left and where subsequent placements of a rook are not restricted, i.e., rooks in the same row are allowed. In the following, this equivalent interpretation will also be used.
\end{remark}

\subsection{Normal ordering and rook numbers}\label{SectNORook}
There exists a close relation between normal ordering in the generalized Weyl algebra and rook numbers. We let $p=q=1$ so that  \eqref{pqcr1} becomes $XY-YX=hY^s$ (for $s \in \mathbb{N}_0$), the commutation relation of the generalized Weyl algebra $A_{s;h|1,1}$ (Definition~\ref{DefpqGenWeyl}, for $p=q=1$). Following Varvak \cite{AV2005}, we associate to a word $\omega$ in $X$ and $Y$ a Ferrers board in $\mathbb{Z}^2$ as follows. Start in $(0,0)$ and represent each letter $X$ as a step to the right, and each letter $Y$ as a step up, reading the word from left to right. The resulting path outlines a Ferrers board $B_{\omega}$.
\begin{example}
The word $XXYXYXXY$ outlines the board $B_{XXYXYXXY}=B_{33211}$ from Figure~\ref{FigFile2}. As another example, the word $(YX)^n=YXYX\cdots YX$ outlines the staircase board $J_{n}$.  
\end{example}
Let us consider a word $\mathrm{w}_{{\bf m}, {\bf n}}=Y^{n_r}X^{m_r}\cdots Y^{n_1}X^{m_1}$ where $XY-YX=hY^s$. To this word we associate a Ferrers board $B_{\mathrm{w}_{{\bf m}, {\bf n}}}$ as described above. In this situation, Varvak \cite[Theorem 7.1]{AV2005} showed that (after correcting a typo) the normal ordered form of $\mathrm{w}_{{\bf m}, {\bf n}}$ is given in terms of $s$-rook numbers by
\begin{equation}\label{RookVarvak0}
\mathrm{w}_{{\bf m}, {\bf n}}=\sum_{j=0}^{|{\bf m}|-m_1} h^j r_j^{(s)}(B_{\mathrm{w}_{{\bf m}, {\bf n}}}) Y^{|{\bf n}|-j(1-s)}X^{|{\bf m}|-j}.
\end{equation}
Using $k=|{\bf m}|-j$, this can be written equivalently as
\begin{equation}\label{RookVarvak}
\mathrm{w}_{{\bf m}, {\bf n}}=\sum_{k=m_1}^{|{\bf m}|} h^{|{\bf m}|-k} r_{|{\bf m}|-k}^{(s)}(B_{\mathrm{w}_{{\bf m}, {\bf n}}}) Y^{|{\bf n}|-(|{\bf m}|-k)(1-s)}X^{k}.
\end{equation}
Comparing \eqref{RookVarvak} with~\eqref{pqrook} yields 
$$
S_{s, h; 1, 1}^{{\bf m}, {\bf n}}[k]=h^{|{\bf m}|-k} r_{|{\bf m}|-k}^{(s)}(B_{\mathrm{w}_{{\bf m}, {\bf n}}}) .
$$
\begin{example}
Let us specialize to $s=0$ and $h=1$, i.e., the Weyl algebra where $XY-YX=I$. If we consider the word $\mathrm{w}_{{\bf m}, {\bf n}}=(YX)^n$, for $n \in \mathbb{N}$, then $B_{\mathrm{w}_{{\bf m}, {\bf n}}}=J_{n}$ and \eqref{RookVarvak} reduces to
$$
(YX)^n=\sum_{k=1}^{n}  r_{n-k}(J_{n}) \, Y^{k}X^{k}=\sum_{k=1}^{n}  S(n,k) \,  Y^{k}X^{k},
$$
where we used in the first step \eqref{RookRed} and in the second step \eqref{StairEq}. If $XY-YX=Y^s$, then $S_{s, 1; 1,1}^{{\bf 1}_r,{\bf 1}_r}[k]=\mathfrak{S}_{s;1}(r,k|1,1)=\mathfrak{S}_{s;1}(r,k)$, and one finds as generalization of the first equation of \eqref{StairEq} that
\begin{equation}\label{SSTirRook}
\mathfrak{S}_{s;1}(n,k)=r_{n-k}^{(s)}(J_n).
\end{equation}
\end{example}
Varvak's interpretation of the numbers $S_{s, h; 1, 1}^{{\bf m}, {\bf n}}[k]$ in terms of $s$-rook numbers was extended by Celeste et al. \cite{CCG2017} to arbitrary $q$ by introducing certain $q$-dependent weights to rook placements. Let us denote a rook placement satisfying the $s$-row creation rule by $\phi$, and the set of all rook placements of $k$ rooks satisfying the $s$-row creation rule on the Ferrers board $B$ by $\mathcal{R}_s(B, k)$. Following \cite{CCG2017}, we define, for $s\neq 0$, the weight $\omega_{s;q}(\phi)$ of a rook placement $\phi \in \mathcal{R}_s(B, k)$ by $\omega_{s;q}(\phi):=h^{u(\phi)}q^{v_s(\phi)}$, where $u(\phi)$ is the number of rooks of the placement, and $v_s(\phi)$ is the number of cells not containing a rook and not lying above a rook. For example, the weight of the rook placement shown in Figure~\ref{FigFile2} is $h^3q^9$. For $s=0$, the parameter $v_0(\phi)$ is the number of cells not containing a rook and not lying above or to the left of a rook (this is the statistic introduced by Garsia and Remmel \cite{AGJR1986}, see also the discussion in \cite{AV2005}). Then the generalized rook number introduced in \cite{CCG2017} is defined by
\begin{equation}\label{RookCeleste}
R_{s, h; q}[B, k] :=\sum\limits_{\phi\in\mathcal{R}_s(B, k)} \omega_{s;q}(\phi)=h^k \sum\limits_{\phi\in\mathcal{R}_s(B, k)} q^{v_s(\phi)}.
\end{equation}
Celeste et al. \cite{CCG2017} showed that one has in this setting 
\begin{equation}\label{qrookfb}
\mathrm{w}_{{\bf m}, {\bf n}}=\sum_{j=0}^{|{\bf{m}}|-m_1} R_{s, h; q}[B_{\mathrm{w}_{{\bf m}, {\bf n}}}, j] \, Y^{|{\bf{n}}|-j(1-s)}X^{|{\bf{m}}|-j}.
\end{equation}
For $q=1$, one has $R_{s, h; 1}[B, j] =h^j |\mathcal{R}_s(B, j)|=h^j r_j^{(s)}(B)$, and \eqref{qrookfb} reduces to \eqref{RookVarvak0}. Comparing \eqref{qrookfb} with~\eqref{pqrook} provides   
$$
S_{s, h; 1, q}^{{\bf m}, {\bf n}}[k]= R_{s, h; q}[B_{\mathrm{w}_{{\bf m}, {\bf n}}}, |{\bf m}|-k].
$$
Similar to \eqref{SSTirRook}, this implies $\mathfrak{S}_{s;1}(n,k|1,q)=R_{s, 1; q}[J_n, n-k]$, for $n\geq 1$. 

\begin{remark}\label{StirFrac}
Recall that the operators $V=X$ and $U=X^{s-1}D$ (with $s\in \mathbb{N}$) act on functions satisfying $UV-VU=V^{s-1}$. Thus, $(X^{s}D)^n=(XX^{s-1}D)^n=(VU)^n$ is related to the generalized Stirling numbers $\mathfrak{S}_{s-1;1}(n,k)$, hence to $(s-1)$-rook numbers of the staircase board $J_n$. Now, one may define ``fractional Stirling numbers'' by considering the expansion of $(X^{\alpha_m}D)^n$, where $\alpha_m = \frac{1}{m}$ with $m \in \mathbb{N}$, see \cite{UK2015} (and \cite{MS2024} for more details on the case $m=2$). Writing $(X^{\alpha_m}D)^n=(XX^{\alpha_m-1}D)^n$, one would expect a connection to $\mathfrak{S}_{\alpha_m-1;1}(n,k)=\mathfrak{S}_{-\frac{m-1}{m};1}(n,k)$. However, if one tries to find a rook interpretation it is not clear how to make sense of the desired ``$-\frac{m-1}{m}$-rook numbers''. Very recently, a combinatorial interpretation was suggested in \cite{DSUK2026}.
\end{remark}
\section{Normal ordering and rook numbers - the $(p,q)$-generalization}\label{spqrook}
\subsection{The $(p,q)$-deformed Weyl algebra ($s=0$)}\label{SectPQRook}
In the case $s=0$, a $(p,q)$-generalization of the weights of rook-placements was introduced for the staircase board $J_n$ in \cite{MWDW1991} and extended to arbitrary Ferrers boards in \cite{JRMW2004}. We follow the style of presentation given in \cite{NLJR2009}. However, we will introduce certain weights for rook placements which are adapted to normal ordering in the $(p,q)$-deformed Weyl algebra $A_{0;h|p,q}$. 

Before starting, note that the new aspect of the $(p,q)$-deformation, for $s=0$,  in~\eqref{pqdefgenWeyl} is that we have a third variable $Z_p$ which ``almost'' commutes with $X$ and $Y$. For the following, we assume that $p\neq 1$ and we abbreviate $Z \equiv Z_p$. As first step to introduce the appropriate weights for rook placements, we introduce a graphical representation of words in the  $(p,q)$-deformed Weyl algebra in terms of an {\it augmented Ferrers board}. Recall how we associated in Section~\ref{SectNORook} to a word $\omega$ in $X$ and $Y$ the Ferrers board $B_{\omega}$ by representing each letter $X$ as a step to the right, and each letter $Y$ as a step up, reading the word from left to right. Now, we have in addition the  letter $Z$ which almost commutes with $X$ and $Y$, i.e., we have 
\begin{equation}\label{ACR}
XZ=p ZX, \hspace{1cm}ZY=p YZ.
\end{equation} 
Let $\omega$ be a word in $X,Y$ and $Z$. We read the word from the left to the right, and, as before, represent each letter $X$ as a step to the right, and each letter $Y$ as a step up. In addition, we represent each power $Z^k$ (for $k\geq 1$) as a circle labelled with $k$ at the corner of the path where it appears. For example, let $\omega = XZ^2XZ^3YZXYXXZ^4Y$, then the associated {\em augmented Ferrers board} is displayed in Figure~\ref{AugmentedBoard} (the path outlining the board is drawn with a heavier weight). 

\begin{figure}[h]
\centering
\begin{tikzpicture}

\draw (0,0.5) -- (0,2);
\draw (0.5,0.65) -- (0.5,2);

\draw[very thick] (1,0.65) -- (1,0.85);
\draw (1,1.15) -- (1,2);   

\draw[very thick] (1.5,1) -- (1.5,1.5);
\draw (1.5,1.5) -- (1.5,2); 

\draw (2,1.5) -- (2,2);

\draw[very thick] (2.5,1.65) -- (2.5,2);  

 
\draw (0,2) -- (2.5,2);

\draw (0,1.5) -- (1.5,1.5);
\draw[very thick] (1.5,1.5) -- (2.35,1.5);

\draw (0,1) -- (0.85,1);
\draw[very thick] (1.15,1) -- (1.5,1); 

\draw[very thick] (0,0.5) -- (0.35,0.5);
\draw[very thick] (0.65,0.5) -- (0.85,0.5);

\draw (0.5,0.5) circle (4pt);
\draw (1,0.5) circle (4pt);

\draw (1,1) circle (4pt);

\draw (2.5,1.5) circle (4pt);

\node at (0.5,0.5) {\footnotesize 2};
\node at (1,0.5) {\footnotesize 3};
\node at (1,1) {\footnotesize 1};
\node at (2.5,1.5) {\footnotesize 4};
\end{tikzpicture}
\caption{Augmented Ferrers board associated to $XZ^2XZ^3YZXYXXZ^4Y$.}\label{AugmentedBoard}
\end{figure}

A word $\omega$ in the letters $X,Y$ and $Z$ is in normal ordered form if the order of the letters is $Y^{l}Z^mX^n$, see Figure~\ref{AugBoardNO}. Recall that the process of normal ordering a word in the conventional Weyl algebra consists of selecting the rightmost corner (i.e., corresponding to a subword $XY$) and either placing a rook (corresponding to $XY \rightsquigarrow h $) or leaving it empty (corresponding to $XY \rightsquigarrow YX $). By iterating over all possibilities of placing the rooks (or not), one obtains the case $s=0$ of \eqref{RookVarvak0}. For the $q$-deformed Weyl algebra the procedure is exactly the same, only after placing the rooks one has to use certain $q$-dependent weights for the rook placements (see Section~\ref{SectNORook}). 

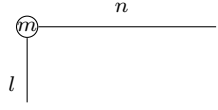
\begin{figure}[h]
\centering
\begin{tikzpicture}
\draw (0,0.5) -- (0,1.35);
\draw (0.15,1.5) -- (2.5,1.5);
\draw (0,1.5) circle (4pt);
\node at (0,1.5) {\footnotesize $m$};

\node at (-0.2,0.75) {\footnotesize $l$};
\node at (1.25,1.75) {\footnotesize $n$};
\end{tikzpicture}
\caption{Augmented Ferrers board associated to the normal ordered word $Y^lZ^mX^n$.}\label{AugBoardNO}
\end{figure}

In the $(p,q)$-deformed Weyl algebra we want to proceed in a similar fashion as in the $q$-deformed variant. However, due to the additional variable $Z$ we represent a word $\omega$ in a first step as an augmented Ferrers board, see Figure~\ref{AugmentedBoard}. Now, consider the right-most corner corresponding to the terminal subword $XZ^4Y$. Clearly, we have to move the letters $Z$ away before we can apply the commutation relation for $X$ and $Y$ (and then either place a rook or not). Since the letter $Z$ almost commutes with $X$ and $Y$, see \eqref{ACR}, moving the letter $Z$ in $\omega$ produces some powers of $p$. Moving a letter $Z$ in $\omega$ corresponds to {\em moving the labelled circle on the path outlined by $\omega$}. In the example of Figure~\ref{AugmentedBoard}, we can move a circle horizontally to the left (i.e., commuting it with $X$, producing according to \eqref{ACR} a factor $p$), or vertically up (i.e., commuting it with $Y$, producing according to \eqref{ACR} also a factor $p$). Thus, all the different possibilities to move the 4 $Z$'s away give the same factor $p^4$. We move all of them to the top, see Figure~\ref{AugmentedBoard2}. Now, the right-most corner is free and we can either place a rook or not.

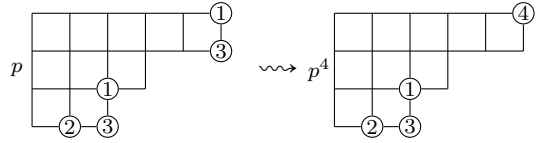
\begin{figure}[h]
\centering
\begin{tikzpicture}

\draw (0,0.5) -- (0,2);
\draw (0.5,0.65) -- (0.5,2);

\draw (1,0.65) -- (1,0.85);
\draw (1,1.15) -- (1,2);   

\draw (1.5,1) -- (1.5,2); 

\draw (2,1.5) -- (2,2);

\draw (2.5,1.65) -- (2.5,1.85);  

 
\draw (0,2) -- (2.35,2);

\draw (0,1.5) -- (2.35,1.5);

\draw (0,1) -- (0.85,1);
\draw (1.15,1) -- (1.5,1); 

\draw (0,0.5) -- (0.35,0.5);
\draw (0.65,0.5) -- (0.85,0.5);

\draw (0.5,0.5) circle (4pt);
\draw (1,0.5) circle (4pt);

\draw (1,1) circle (4pt);
\draw (2.5,1.5) circle (4pt);
\draw (2.5,2) circle (4pt);

\node at (0.5,0.5) {\footnotesize 2};
\node at (1,0.5) {\footnotesize 3};
\node at (1,1) {\footnotesize 1};
\node at (2.5,1.5) {\footnotesize 3};
\node at (2.5,2) {\footnotesize 1};

\node at (-0.2,1.25) {\footnotesize $p$};

\node at (3.25,1.25) {$\rightsquigarrow$};

\draw (4,0.5) -- (4,2);
\draw (4.5,0.65) -- (4.5,2);

\draw (5,0.65) -- (5,0.85);
\draw (5,1.15) -- (5,2);   

\draw (5.5,1) -- (5.5,2); 

\draw (6,1.5) -- (6,2);

\draw (6.5,1.5) -- (6.5,1.85);  

 
\draw (4,2) -- (6.35,2);

\draw (4,1.5) -- (6.5,1.5);

\draw (4,1) -- (4.85,1);
\draw (5.15,1) -- (5.5,1); 

\draw (4,0.5) -- (4.35,0.5);
\draw (4.65,0.5) -- (4.85,0.5);

\draw (4.5,0.5) circle (4pt);
\draw (5,0.5) circle (4pt);

\draw (5,1) circle (4pt);

\draw (6.5,2) circle (4pt);

\node at (4.5,0.5) {\footnotesize 2};
\node at (5,0.5) {\footnotesize 3};
\node at (5,1) {\footnotesize 1};
\node at (6.5,2) {\footnotesize 4};

\node at (3.8,1.25) {\footnotesize $p^4$};

\end{tikzpicture}
\caption{The augmented Ferrers board from Figure~\ref{AugmentedBoard} after moving one of the right-most circles to the top (left), and after moving all of them to the top (right).}\label{AugmentedBoard2}
\end{figure}

\begin{example}\label{NOEX}
Let us consider a very simple example for normal ordering a word and its representation using augmented Ferrers boards. Let $\omega=(YX)^3=YXYXYX$. Recall that each time we have a corner, i.e., a subword $XY$, we can either (1) commute $XY \rightsquigarrow qYX$, or (2) contract $XY \rightsquigarrow hZ$. We label the augmented Ferrers board by the sequence of operations to generate it. For example, the second board in the second row in Figure~\ref{ABNO} is labelled with $(12)$ since we first commuted and then contracted. 

\begin{figure}[h]
\centering
\begin{tikzpicture}

\draw (0,0.5) -- (0,2);
\draw (0.5,1) -- (0.5,2);
\draw (1,1.5) -- (1,2);
\draw (0,2) -- (1.5,2);
\draw (0,1.5) -- (1,1.5);
\draw (0,1) -- (0.5,1);
\node at (0.75,0.25) {\footnotesize $YXYXYX$};

\node at (2.25,1.25) {$\rightsquigarrow$};

\node at (2.75,1.25) {$q $};
\draw (3,0.5) -- (3,2);
\draw (3.5,1) -- (3.5,2);
\draw (3,2) -- (4.5,2);
\draw (3,1.5) -- (3.5,1.5);
\draw (3,1) -- (3.5,1);

\node at (4.75,2) {$(1)$};
\node at (3.75,0.25) {\footnotesize $qYXY^2X^2$};

\node at (5.25,1.25) {$+$};

\node at (5.75,1.25) {$h$};
\draw (6,1) -- (6,2);
\draw (6.5,1.5) -- (6.5,1.85);
\draw (6,2) -- (6.35,2);
\draw (6.65,2) -- (7,2);
\draw (6,1.5) -- (6.5,1.5);
\node at (7.25,2) {$(2)$};

\draw (6.5,2) circle (4pt);
\node at (6.5,2) {\footnotesize 1};

\node at (6.75,0.25) {\footnotesize $hYXYZX$};


\node at (2.25,-1.75) {$\rightsquigarrow$};

\node at (2.75,-1.75) {$q^2 $};
\draw (3,-2.5) -- (3,-1);
\draw (3.5,-1.5) -- (3.5,-1);
\draw (3,-1) -- (4.5,-1);
\draw (3,-1.5) -- (3.5,-1.5);
\node at (4.9,-1) {$(11)$};

\node at (3.75,-2.75) {\footnotesize $q^2Y^2XYX^2$};

\node at (5.25,-1.75) {$+$};

\node at (5.75,-1.75) {$qh$};
\draw (6,-2) -- (6,-1.65);
\draw (6,-1.35) -- (6,-1);
\draw (6,-1) -- (7,-1);
\node at (7.4,-1) {$(12)$};

\draw (6,-1.5) circle (4pt);
\node at (6,-1.5) {\footnotesize 1};

\node at (6.75,-2.75) {\footnotesize $qhYZYX^2$};

\node at (7.75,-1.75) {$+$};

\node at (8.75,-1.75) {$qh$};
\draw (9,-2) -- (9,-1);
\draw (9,-1) -- (9.35,-1);
\draw (9.65,-1) -- (10,-1);
\node at (10.4,-1) {$(21)$};

\draw (9.5,-1) circle (4pt);
\node at (9.5,-1) {\footnotesize 1};

\node at (9.75,-2.75) {\footnotesize $qhY^2XZX$};

\node at (10.75,-1.75) {$+$};

\node at (11.75,-1.75) {$h^2$};
\draw (12,-1.5) -- (12,-1.15);
\draw (12.15,-1) -- (12.5,-1);
\node at (12.9,-1) {$(22)$};

\draw (12,-1) circle (4pt);
\node at (12,-1) {\footnotesize 2};

\node at (12.75,-2.75) {\footnotesize $h^2YZ^2X$};


\node at (2.25,-4.75) {$\rightsquigarrow$};

\node at (2.75,-4.75) {$q^3 $};
\draw (3,-5.5) -- (3,-4);
\draw (3,-4) -- (4.5,-4);
\node at (5,-4) {$(111)$};

\node at (3.75,-5.75) {\footnotesize $q^3Y^3X^3$};

\node at (5.25,-4.75) {$+$};

\node at (5.65,-4.75) {$q^2h$};
\draw (6,-5) -- (6,-4.15);
\draw (6.15,-4) -- (7,-4);
\node at (7.5,-4) {$(112)$};

\draw (6,-4) circle (4pt);
\node at (6,-4) {\footnotesize 1};

\node at (6.75,-5.75) {\footnotesize $q^2h Y^2ZX^2$};

\node at (10,-4.75) {$+ (12) + (21) + (22)$};

\end{tikzpicture}
\caption{Normal ordering the word $YXYXYX$ using augmented Ferrers boards.}\label{ABNO}
\end{figure}
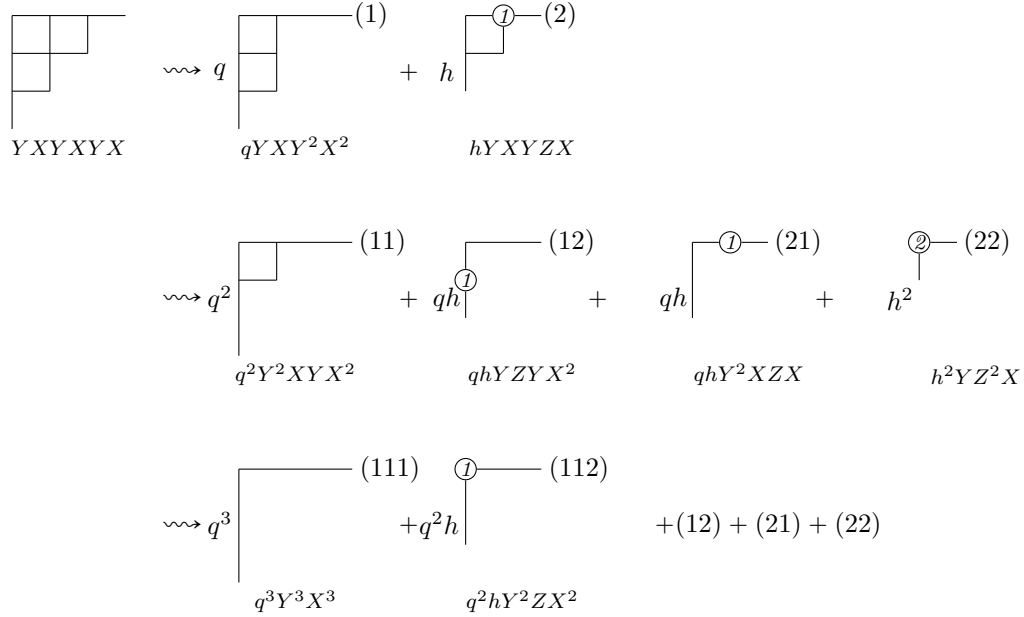
In Figure~\ref{ABNO}, in the third row the nearly final result is shown. Commuting the letter $Z$ in the boards $(12)$ and $(21)$ to its middle position will give for each board an additional factor of $p$. Thus,
\begin{equation}\label{ABNOEQ}
(YX)^3=q^3Y^3X^3 + \left(q^2 + pq + pq\right) h Y^2ZX^2  +h^2YZ^2X.´
\end{equation}
Note that each resulting final augmented board can be related in a one-to-one way to a Ferrers board where some rooks are placed (identification by the sequence of commutations and contractions).  In Figure~\ref{ABNO2}, the possible rook placements on the Ferrers board for $(YX)^3$ can be seen, together with a reference to the augmented Ferrers board and the corresponding weight, the rook placement must have to recover the normal ordering result \eqref{ABNOEQ}.

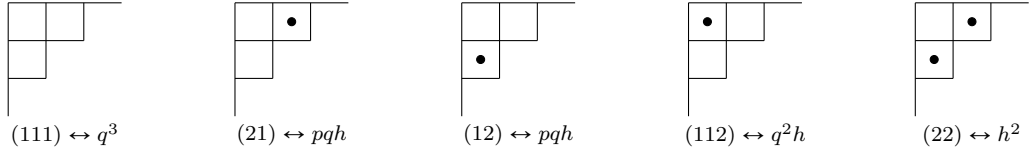
\begin{figure}[h]
\centering
\begin{tikzpicture}

\draw (0,0.5) -- (0,2);
\draw (0.5,1) -- (0.5,2);
\draw (1,1.5) -- (1,2);
\draw (0,2) -- (1.5,2);
\draw (0,1.5) -- (1,1.5);
\draw (0,1) -- (0.5,1);
\node at (0.75,0.25) {\footnotesize $(111)\leftrightarrow q^3$};

\draw (3,0.5) -- (3,2);
\draw (3.5,1) -- (3.5,2);
\draw (4,1.5) -- (4,2);
\draw (3,2) -- (4.5,2);
\draw (3,1.5) -- (4,1.5);
\draw (3,1) -- (3.5,1);
\filldraw (3.75,1.75) circle (1.5pt);
\node at (3.75,0.25) {\footnotesize $(21)\leftrightarrow pqh$};

\draw (6,0.5) -- (6,2);
\draw (6.5,1) -- (6.5,2);
\draw (7,1.5) -- (7,2);
\draw (6,2) -- (7.5,2);
\draw (6,1.5) -- (7,1.5);
\draw (6,1) -- (6.5,1);
\filldraw (6.25,1.25) circle (1.5pt);
\node at (6.75,0.25) {\footnotesize $(12)\leftrightarrow pqh$};

\draw (9,0.5) -- (9,2);
\draw (9.5,1) -- (9.5,2);
\draw (10,1.5) -- (10,2);
\draw (9,2) -- (10.5,2);
\draw (9,1.5) -- (10,1.5);
\draw (9,1) -- (9.5,1);
\filldraw (9.25,1.75) circle (1.5pt);
\node at (9.75,0.25) {\footnotesize $(112)\leftrightarrow q^2h $};

\draw (12,0.5) -- (12,2);
\draw (12.5,1) -- (12.5,2);
\draw (13,1.5) -- (13,2);
\draw (12,2) -- (13.5,2);
\draw (12,1.5) -- (13,1.5);
\draw (12,1) -- (12.5,1);
\filldraw (12.25,1.25) circle (1.5pt);
\filldraw (12.75,1.75) circle (1.5pt);
\node at (12.75,0.25) {\footnotesize $(22)\leftrightarrow h^2$};

\end{tikzpicture}
\caption{Rook placements on the Ferrers board of $YXYXYX$.}\label{ABNO2}
\end{figure}
\end{example}
The above example shows that a contraction (i.e., rook placement on the Ferrers board) yields a letter $Z$ and, hence, a labelled circle on the augmented Ferrers board. This has to be pushed to the origin -- like in the boards (22) or (112) --, giving some powers of $p$. To determine the corresponding power of $p$, we first observe that when placing the rook, all boxes lying above or to the left of this rook will give a factor of $p$ -- but with two exceptions: 
\begin{itemize}
\item A box lying above the rook and to the left of another rook does not contribute a factor of $p$ since this row is canceled by the other rook to the right.
\item A box lying to the left of the rook and in the column of another rook does not contribute a factor of $p$ since this column is cancelled by the other rook to the left.    
\end{itemize} 
We now formalize the above observations and introduce a $(p,q)$-weight for rook placements which generalizes the one for $p=1$ from the preceding section (for $s=0$). Let $\phi$ be a placement of $k$ non-attacking rooks on the Ferrrers board $B$. We partition the boxes of $B$ into 4 disjoint classes:
\begin{enumerate}
\item A {\em rook box} is a box of $B$ occupied by a rook.
\item A {\em $p$-box} is a box which is not a rook box and which is lying above or to the left of a rook, and
\begin{itemize}
\item it is not lying to the left of another rook, and
\item it is not lying in the column of another rook.
\end{itemize}
\item A {\em cancelled box} is a box located above or to the left of a rook which is not a $p$-box.
\item All remaining boxes of $B$ are {\em $q$-boxes}.
\end{enumerate} 
\begin{definition}
The {\em $(p,q)$-weight} of a placement $\phi$ of $k$ non-attacking rooks on the Ferrers board $B$ is defined to be
\begin{equation}\label{PQWeight}
\omega_{p,q}(B;\phi) :=h^k p^{\# p-boxes}q^{\# q-boxes}.
\end{equation}
\end{definition}
It is easy to see that this definition gives the correct weights for the rook placements displayed in Figure~\ref{ABNO2}. Let us give a more complex example. In the following figure, we mark a rook box by a rook, a $p$-box (respectively, $q$-box) with a $p$ (respectively, $q$), and the cancelled boxes with an $x$. In Figure~\ref{FigFilePQ} a 6-rook placement on a board can be seen. There are in total 74 boxes: 6 rook-boxes, 16 $p$-boxes, 15 cancelled boxes, and 37 $q$-boxes. Thus, the $(p,q)$-weight of this placement is given by $h^6 p^{16}q^{37}$.

\begin{figure}[h]
\centering
\begin{tikzpicture}

\draw (0,6.5) -- (6,6.5);
\draw (0,6) -- (6,6);
\draw (0,5.5) -- (5,5.5);
\draw (0,5) -- (5,5);
\draw (0,4.5) -- (4,4.5); 
\draw (0,4) -- (4,4); 
\draw (0,3.5) -- (3.5,3.5); 
\draw (0,3) -- (2.5,3); 
\draw (0,2.5) -- (2.5,2.5); 
\draw (0,2) -- (2.5,2); 
\draw (0,1.5) -- (1,1.5); 
\draw (0,1) -- (1,1);


\draw (0,1) -- (0,6.5);
\draw (0.5,1) -- (0.5,6.5);
\draw (1,1) -- (1,6.5);
\draw (1.5,2) -- (1.5,6.5);
\draw (2,2) -- (2,6.5);
\draw (2.5,2) -- (2.5,6.5);
\draw (3,3.5) -- (3,6.5);
\draw (3.5,3.5) -- (3.5,6.5);
\draw (4,4) -- (4,6.5);
\draw (4.5,5) -- (4.5,6.5);
\draw (5,5) -- (5,6.5);
\draw (5.5,6) -- (5.5,6.5);
\draw (6,6) -- (6,6.5);

\node at (0.25,6.25) {x};
\node at (0.75,6.25) {x};
\node at (1.25,6.25) {x};
\node at (1.75,6.25) {$p$};
\node at (2.25,6.25) {x};
\node at (2.75,6.25) {$p$};
\node at (3.25,6.25) {x};
\node at (3.75,6.25) {$p$};
\node at (4.25,6.25) {$p$};
\node at (4.75,6.25) {$p$};
\filldraw (5.25,6.25) circle (3pt);
\node at (5.75,6.25) {$q$};

\node at (0.25,5.75) {x};
\node at (0.75,5.75) {x};
\node at (1.25,5.75) {x};
\node at (1.75,5.75) {$p$};
\node at (2.25,5.75) {x};
\node at (2.75,5.75) {$p$};
\filldraw (3.25,5.75) circle (3pt);
\node at (3.75,5.75) {$q$};
\node at (4.25,5.75) {$q$};
\node at (4.75,5.75) {$q$};

\node at (0.25,5.25) {$p$};
\node at (0.75,5.25) {$p$};
\node at (1.25,5.25) {$p$};
\node at (1.75,5.25) {$q$};
\node at (2.25,5.25) {$p$};
\node at (2.75,5.25) {$q$};
\node at (3.25,5.25) {$q$};
\node at (3.75,5.25) {$q$};
\node at (4.25,5.25) {$q$};
\node at (4.75,5.25) {$q$};

\filldraw (0.25,4.75) circle (3pt);
\node at (0.75,4.75) {$p$};
\node at (1.25,4.75) {$p$};
\node at (1.75,4.75) {$q$};
\node at (2.25,4.75) {$p$};
\node at (2.75,4.75) {$q$};
\node at (3.25,4.75) {$q$};
\node at (3.75,4.75) {$q$};

\node at (0.25,4.25) {x};
\node at (0.75,4.25) {x};
\node at (1.25,4.25) {x};
\node at (1.75,4.25) {$p$};
\filldraw (2.25,4.25) circle (3pt);
\node at (2.75,4.25) {$q$};
\node at (3.25,4.25) {$q$};
\node at (3.75,4.25) {$q$};

\node at (0.25,3.75) {x};
\node at (0.75,3.75) {x};
\filldraw (1.25,3.75) circle (3pt);
\node at (1.75,3.75) {$q$};
\node at (2.25,3.75) {$q$};
\node at (2.75,3.75) {$q$};
\node at (3.25,3.75) {$q$};

\node at (0.25,3.25) {$q$};
\node at (0.75,3.25) {$p$};
\node at (1.25,3.25) {$q$};
\node at (1.75,3.25) {$q$};
\node at (2.25,3.25) {$q$};

\node at (0.25,2.75) {x};
\filldraw (0.75,2.75) circle (3pt);
\node at (1.25,2.75) {$q$};
\node at (1.75,2.75) {$q$};
\node at (2.25,2.75) {$q$};

\node at (0.25,2.25) {$q$};
\node at (0.75,2.25) {$q$};
\node at (1.25,2.25) {$q$};
\node at (1.75,2.25) {$q$};
\node at (2.25,2.25) {$q$};

\node at (0.25,1.75) {$q$};
\node at (0.75,1.75) {$q$};

\node at (0.25,1.25) {$q$};
\node at (0.75,1.25) {$q$};

\end{tikzpicture}
\caption{A Ferrers board with a $6$-rook placement and marking of the boxes according to their class.}\label{FigFilePQ}
\end{figure}

Now, we can define the {\em $(p,q)$-deformed rook numbers} as follows.
\begin{definition}
Let $B$ be a Ferrers board and let $\mathcal{R}(B, k)$ be the set of placements of $k$ non-attacking rooks on $B$. Then the {\em $(p,q)$-deformed rook numbers} are defined by
\begin{equation}\label{PQRookWe}
\mathscr{R}_{h; p,q}[B, k] :=\sum\limits_{\phi\in\mathcal{R}(B, k)} \omega_{p,q}(B;\phi)=h^k \sum\limits_{\phi\in\mathcal{R}(B, k)} p^{\# p-boxes}q^{\# q-boxes}.
\end{equation}
\end{definition}
Note that the definition of a $q$-box in this case means that it is neither a rook-box nor is it located to the left or above a rook-box. Thus, the number of $q$-boxes in the weight $\omega_{p,q}(B;\phi)$ defined in \eqref{PQWeight} is exactly the parameter $v_0(\phi)$ of the preceding section. Thus, for $p=1$, the weight defined here reduces to the one of the preceding section, $\omega_{1,q}(B;\phi)=\omega_{0;q}(\phi)$. Consequently, the $(p,q)$-deformed rook numbers defined here reduce for $p=1$ to the ones considered in \eqref{RookCeleste} (for $s=0$),
$$
\mathscr{R}_{h;1,q}[B, k]=R_{0,h;q}[B, k].
$$
In general, we have the following result.
\begin{theorem}\label{TheoWordNO}
Let $\mathrm{w}_{{\bf m}, {\bf n}}=Y^{n_r}X^{m_r}\cdots Y^{n_1}X^{m_1}$ be a word in the letters $X,Y$ satisfying the commutation relations \eqref{pqdefgenWeyl}, for $s=0$, of the $(p,q)$-deformed Weyl algebra $A_{0;h|p,q}$. Then one has the normal ordering result
\begin{equation}\label{pqrookNO}
\mathrm{w}_{{\bf m}, {\bf n}}=\sum_{k=0}^{|{\bf m}| -m_1} \mathscr{R}_{h; p,q}[B_{\mathrm{w}_{{\bf m}, {\bf n}}}, k] \, Y^{|{\bf n}|-k}Z_p^{k}X^{|{\bf m}|-k},
\end{equation} 
where $B_{\mathrm{w}_{{\bf m}, {\bf n}}}$ is the Ferrers board associated to $\mathrm{w}_{{\bf m}, {\bf n}}$.
\end{theorem}
Letting $\ell = |{\bf m}|-k$ in \eqref{pqrookNO}, a comparison with \eqref{pqrook} shows that 
\begin{equation}\label{Identification}
S_{0, h; p, q}^{{\bf m}, {\bf n}}[\ell]=\mathscr{R}_{h; p,q}[B_{\mathrm{w}_{{\bf m}, {\bf n}}}, |{\bf m}|-\ell] .
\end{equation}

\begin{corollary}\label{PQStirlingRook}
Let us consider the word $\omega=(YX)^n$, for $n \in \mathbb{N}$. Since $B_{(YX)^n}=J_{n}$, \eqref{pqrookNO} gives
\begin{equation}\label{Equ:PQSTirlingRook}
	(YX)^n=\sum_{\ell=0}^{n} \mathscr{R}_{h; p,q}[J_{n}, n-\ell ] \, Y^{\ell }Z_p^{n-\ell}X^{\ell}.
\end{equation}
Setting $s=0$ in \eqref{StirGen} yields
 \begin{equation}\label{StirGenS0}
	(YX)^n=\sum_{k=0}^n \mathfrak{S}_{0;h}(n,k|p,q) \, Y^{k}Z_p^{n-k} X^{k}.
\end{equation}
Comparing \eqref{Equ:PQSTirlingRook} with \eqref{StirGenS0}, we find $\mathfrak{S}_{0;h}(n,\ell|p,q)=\mathscr{R}_{h; p,q}[J_{n}, n-\ell ]$, for $n \geq 1$. According to~\cite[Proposition 3]{LO2020}, one can identify $\mathfrak{S}_{0;1}(n,\ell|p,q)=S_{p,q}(n,\ell)$ with the $(p,q)$-Stirling numbers of the second kind discussed in \cite{LO2020}, where $S_{p,q}(0,0)=0$ and 
\begin{equation}\label{pqStirlingRecu}
	S_{p,q}(n,k)=p^{n-k}q^{k-1}S_{p,q}(n-1,k-1)+[k]_{p,q}S_{p,q}(n-1,k),
\end{equation}
follows also from \eqref{Stirlingrecu} for $s=0$ and $h=1$. Hence 
\begin{equation}\label{PQStirlingRookEq}
	S_{p,q}(n,\ell)=\mathscr{R}_{1; p,q}[J_{n}, n-\ell ].
\end{equation} 
This gives the $(p,q)$-deformed Stirling numbers $S_{p,q}(n,\ell)$ -- defined in terms of normal ordering -- a combinatorial interpretation in terms of $(p,q)$-deformed rook numbers. 
\end{corollary}

What about considering the letter $Z_p$ on the left-hand side? Since $Z_p$ almost commutes with $X$ and $Y$ we can pull all letters $Z_p$ to the right (or, alternatively, to the left), yielding some powers of $p$ and something of the form $\omega^{\ast} Z_p^n$, where $\omega^{\ast}$ is a word in $X$ and $Y$ only. This can be normal ordered and then the factor $Z_p^n$ be brought to its correct position. The following argument until \eqref{intermediate} is exactly the same as in \cite{TLM1}, reproduced here to be self-contained. Let us consider this in more detail and let 
\begin{equation}\label{PQWeylWord}
\mathrm{w}_{{\bf m},{\bf n},{\bf u}}=Z_p^{u_{2r}}Y^{n_r}Z_p^{u_{2r-1}}X^{m_r}\cdots Z_p^{u_4}Y^{n_2}Z_p^{u_3}X^{m_2}\cdot Z_p^{u_2}Y^{n_1}Z_p^{u_1}X^{m_1}\cdot Z_p^{u_0}
\end{equation}
be an arbitrary word in $X,Y$ and $Z_p$. Except for the right-most factor $Z_p^{u_0}$ we can group the factors as $r$ ``blocks'' of the form $Z_p^{u_{2k}}Y^{n_k}Z_p^{u_{2k-1}}X^{m_k}$, where $k=1,\ldots,r$. Commuting $Z_p^{u_{2k-1}}$ to the right yields a factor $p^{u_{2k-1}(\sum_{j=1}^{k-1}n_j-\sum_{j=1}^{k}m_j)}$, whereas commuting $Z_p^{u_{2k}}$ to the right yields a factor $p^{u_{2k}(\sum_{j=1}^{k}n_j-\sum_{j=1}^{k}m_j)}$. Together, we get as exponent 
$$
u_{2k-1}\left(\sum_{j=1}^{\lfloor \frac{2k-1}{2} \rfloor} n_j-\sum_{j=1}^{\lfloor \frac{2k}{2} \rfloor}m_j\right)+u_{2k}\left(\sum_{j=1}^{\lfloor \frac{2k}{2} \rfloor}n_j-\sum_{j=1}^{\lfloor \frac{2k+1}{2} \rfloor}m_j)\right).
$$
Thus, pulling all letters $Z_p$ of $\mathrm{w}_{{\bf m},{\bf n},{\bf u}}$ to the right, we get a factor of $p^{\mathscr{L}({\bf m},{\bf n},{\bf u})}$, where
\begin{equation}\label{PQExponent}
\mathscr{L}({\bf m},{\bf n},{\bf u}):=\sum_{\ell =1}^{2r} u_{\ell }\left(\sum_{j=1}^{\lfloor \frac{\ell }{2} \rfloor} n_j-\sum_{j=1}^{\lfloor \frac{\ell +1}{2} \rfloor}m_j\right).
\end{equation}
Denoting by $\mathrm{w}_{{\bf m},{\bf n},{\bf u}}^{\ast}$ the {\it reduced word} of $\mathrm{w}_{{\bf m},{\bf n},{\bf u}}$ where all letters $Z_p$ are deleted from $\mathrm{w}_{{\bf m},{\bf n},{\bf u}}$, we can write (using  $|{\bf u}|=u_{2r}+u_{2r-1}+\cdots + u_1+u_0$)
\begin{equation}\label{intermediate}
\mathrm{w}_{{\bf m},{\bf n},{\bf u}}=p^{\mathscr{L}({\bf m},{\bf n},{\bf u})}\mathrm{w}_{{\bf m},{\bf n},{\bf u}}^{\ast} Z_p^{|{\bf u}|}.
\end{equation}
Now, we can use \eqref{pqrookNO} for $\mathrm{w}_{{\bf m},{\bf n},{\bf u}}^{\ast}$ to find 
$$
\mathrm{w}_{{\bf m},{\bf n},{\bf u}}=p^{\mathscr{L}({\bf m},{\bf n},{\bf u})}\sum_{k=0}^{|{\bf m}|-m_1} \mathscr{R}_{h; p,q}[B_{\mathrm{w}_{{\bf m},{\bf n},{\bf u}}^{\ast}}, k] \, Y^{|{\bf n}|-k}Z_p^{k}X^{|{\bf m}|-k} Z_p^{|{\bf u}|},
$$
where $B_{\mathrm{w}_{{\bf m},{\bf n},{\bf u}}^{\ast}}$ is the Ferrers board associated to $\mathrm{w}_{{\bf m},{\bf n},{\bf u}}^{\ast}$.
Thus, we have shown the following proposition.
\begin{proposition}\label{PropPQWeylNO}
Let $\mathrm{w}_{{\bf m},{\bf n},{\bf u}}$ be an arbitrary word in the letters $X,Y$ and $Z_p$ satisfying the commutation relations \eqref{pqdefgenWeyl}, for $s=0$, of the $(p,q)$-deformed Weyl algebra $A_{0;h|p,q}$. Then one has the normal ordering result
\begin{equation}\label{pqrookNOG}
\mathrm{w}_{{\bf m},{\bf n},{\bf u}}=\sum_{k=0}^{|{\bf m}|-m_1} p^{|{\bf u}|(|{\bf m}|-k)+\mathscr{L}({\bf m},{\bf n},{\bf u})} \mathscr{R}_{h; p,q}[B_{\mathrm{w}_{{\bf m},{\bf n},{\bf u}}^{\ast}}, k] \, Y^{|{\bf n}|-k}Z_p^{|{\bf u}|+k}X^{|{\bf m}|-k},
\end{equation}
where $\mathrm{w}_{{\bf m},{\bf n},{\bf u}}^{\ast}$ denotes the reduced word of $\mathrm{w}_{{\bf m},{\bf n},{\bf u}}$ and $\mathscr{L}({\bf m},{\bf n},{\bf u})$ is defined in \eqref{PQExponent}.
\end{proposition}
Clearly, if no letters $Z_p$ appear in $\mathrm{w}_{{\bf m},{\bf n},{\bf u}}$, then \eqref{pqrookNOG} reduces to \eqref{pqrookNO}.

The conventional Stirling numbers of the second kind have a combinatorial interpretation as rook numbers of the staircase board, $S(n,k)=r_{n-k}(J_n)$, see \eqref{StairEq}, and the same holds true in the $(p,q)$-deformed case, see \eqref{PQStirlingRookEq}. This identification was slightly indirect by comparing normal ordering coefficients. We now derive the recurrence relation of the $(p,q)$-deformed Stirling numbers using the combinatorial interpretation. Since the proof follows the argument of the conventional case we repeat this argument here first. Since $S(n+1,k)=r_{n-k+1}(J_{n+1})$, we want to count the possibilities to place $(n-k+1)$ non-attacking rooks on the staircase board $J_{n+1}$. Note that we can generate $J_{n+1}$ from $J_n$ by adjoining a column of length $n$ to the left of $J_{n}$. Now, we can place the $(n-k+1)$ rooks in two ways:
\begin{itemize}
\item Type I: We leave the first column empty and place all $(n-k+1)$ rooks in $J_n$. There are $r_{n-k+1}(J_{n})$ possibilities.
\item Type II: We put $(n-k)$ rooks in $J_n$ and put the remaining rook into one of the $k$ possible boxes (the rooks have to be non-attacking). Thus, there exist $k \cdot r_{n-k}(J_{n})$ possibilities.
\end{itemize}  
Thus, in total we have 
$$S(n+1,k)=r_{n-k+1}(J_{n+1})=r_{n-k+1}(J_{n})+k r_{n-k}(J_{n}) =S(n,k-1)+k S(n,k),
$$
as to be shown. Let us turn to the $(p,q)$-deformed case, and let us define 
$$
T_{p,q}(n,k):= \mathscr{R}_{1; p,q}[J_n,n- k].
$$
\begin{proposition}\label{PQPropStirRook} The numbers $T_{p,q}(n,k)$ satisfy $T_{p,q}(0,0)=1$ and the recurrence relation 
$$
T_{p,q}(n+1,k)=p^{n-k+1}q^{k-1}T_{p,q}(n,k-1)+[k]_{p,q}T_{p,q}(n,k).
$$
\end{proposition}
\begin{proof}
The argument is similar to the conventional case repeated above since the set of rook placements coincides. In the $(p,q)$-deformed case we have to consider the weights of the rook placements. Let us consider $T_{p,q}(n+1,k)= \mathscr{R}_{1; p,q}[J_{n+1},n- k+1]$, i.e., the set of $(n-k+1)$ rook placements on $J_{n+1}$. 

We start with the rook placements of Type I. Let $\phi$ be such a placement. Considered as a placement on $J_{n}$ it has weight $\omega_{p,q}(J_n;\phi)$. However, when considered as a placement on $J_{n+1}$ it has weight $\omega_{p,q}(J_{n+1};\phi)=p^{n-k+1}q^{k-1}\omega_{p,q}(J_n;\phi)$ since in the adjoined column with $n$ boxes one has $(n-k+1)$ $p$-boxes (the rows of the rooks) and $(k-1)$ $q$-boxes (the rows without a rook). Since this holds for all placements of type I, their contribution is given by
$$
p^{n-k+1}q^{k-1}\sum\limits_{\phi\in\mathcal{R}(J_n, n-k+1)} \omega_{p,q}(J_n;\phi)=p^{n-k+1}q^{k-1}\mathscr{R}_{1; p,q}[J_n,n- k+1]=p^{n-k+1}q^{k-1}T_{p,q}(n,k-1).
$$
Now, let us turn to the placements of type II. Let $\phi'$ be a placement of $(n-k)$ rooks on $J_n$ with weight $\omega_{p,q}(J_n;\phi')$. For each such placement we obtain $k$ placements on $J_{n+1}$ by choosing for the $(n-k+1)$-th rook one of the $k$ possible ``free'' boxes. Placing this rook into the first free box from above, there will result no additional $p$-boxes and $(k-1)$ additional $q$-boxes (all free boxes located below this rook). Thus, the weight of this placement on $J_{n+1}$ will be $p^0q^{k-1}\omega_{p,q}(J_n;\phi')$. In general, placing this rook into the $r$-th free box from above (where $r=1,\ldots,k$) will result in $(r-1)$ additional $p$-boxes and $(k-r+1)$ additional $q$-boxes, implying for this placement the weight $p^{r-1}q^{k-r+1}\omega_{p,q}(J_n;\phi')$. Thus, summing the contribution of the weights of all $k$ placements on $J_{n+1}$ resulting from the placement $\phi'$, one finds
$$
(p^0q^{k-1}+p^1q^{k-2}+\cdots+p^{k-2}q^{1}+p^{k-1}q^{0})\,\omega_{p,q}(J_n;\phi')=[k]_{p,q}\,\omega_{p,q}(J_n;\phi').
$$
Since this holds for all placements of type II, their contribution is given by
$$
[k]_{p,q}\sum\limits_{\phi'\in\mathcal{R}(J_n, n-k)} \omega_{p,q}(J_n;\phi')=[k]_{p,q}\mathscr{R}_{1; p,q}[J_n,n- k]=[k]_{p,q}T_{p,q}(n,k).
$$
Adding the contributions from placements of type I and type II yields the assertion. 
\end{proof}
Comparing the above recurrence relation (and initial value) for $T_{p,q}(n,k)$ with the one given in \eqref{pqStirlingRecu} for $S_{p,q}(n,k)$, we see that $T_{p,q}(n,k)=S_{p,q}(n,k)$. Thus, the $(p,q)$-deformed Stirling numbers $S_{p,q}(n,k)$ are given by the $(p,q)$-deformed rook numbers of the staircase board $J_n$, i.e., $S_{p,q}(n,k)=\mathscr{R}_{1; p,q}[J_n,n- k]$, as already mentioned in \eqref{PQStirlingRookEq}.

\subsection{The case $s>0$}\label{SectPQRookS}
The case $s>0$ is similar to the case $s=0$ considered in the preceding section. In particular, one may associate to a word in the letters $X,Y$ and $Z$ an augmented Ferrers board in exactly the same fashion as described in Section~\ref{SectPQRook}. However, there is now a difference when normal ordering. As in the case $s=0$, when selecting the right-most corner of the board (corresponding to a subword $XY$) one has the possibility (1) to commute (corresponding to $XY \rightsquigarrow qYX$), or (2) to contract (corresponding to $XY \rightsquigarrow hY^s Z$). In contrast to the contraction in the case $s=0$ (which resulted in the cancellation of a column and a row) the contraction will result here in the cancellation of a column and the addition of $s$ boxes in each of the columns to the left.   

To be more concrete, we consider the word $\omega=(YX)^3$ -- discussed in Example~\ref{NOEX} for the $(p,q)$-deformed Weyl algebra -- again for the present context in the next example.

\begin{example}\label{NOEXS}
Let $s=2$ and let $\omega=(YX)^3=YXYXYX$. Recall that each time we have a corner, i.e., a subword $XY$, we can either (1) commute $XY \rightsquigarrow qYX$, or (2) contract $XY \rightsquigarrow hY^2 Z$. In Figure~\ref{ABNOS}, the normal ordering of $YXYXYX$ is described in terms augmented Ferrers boards.

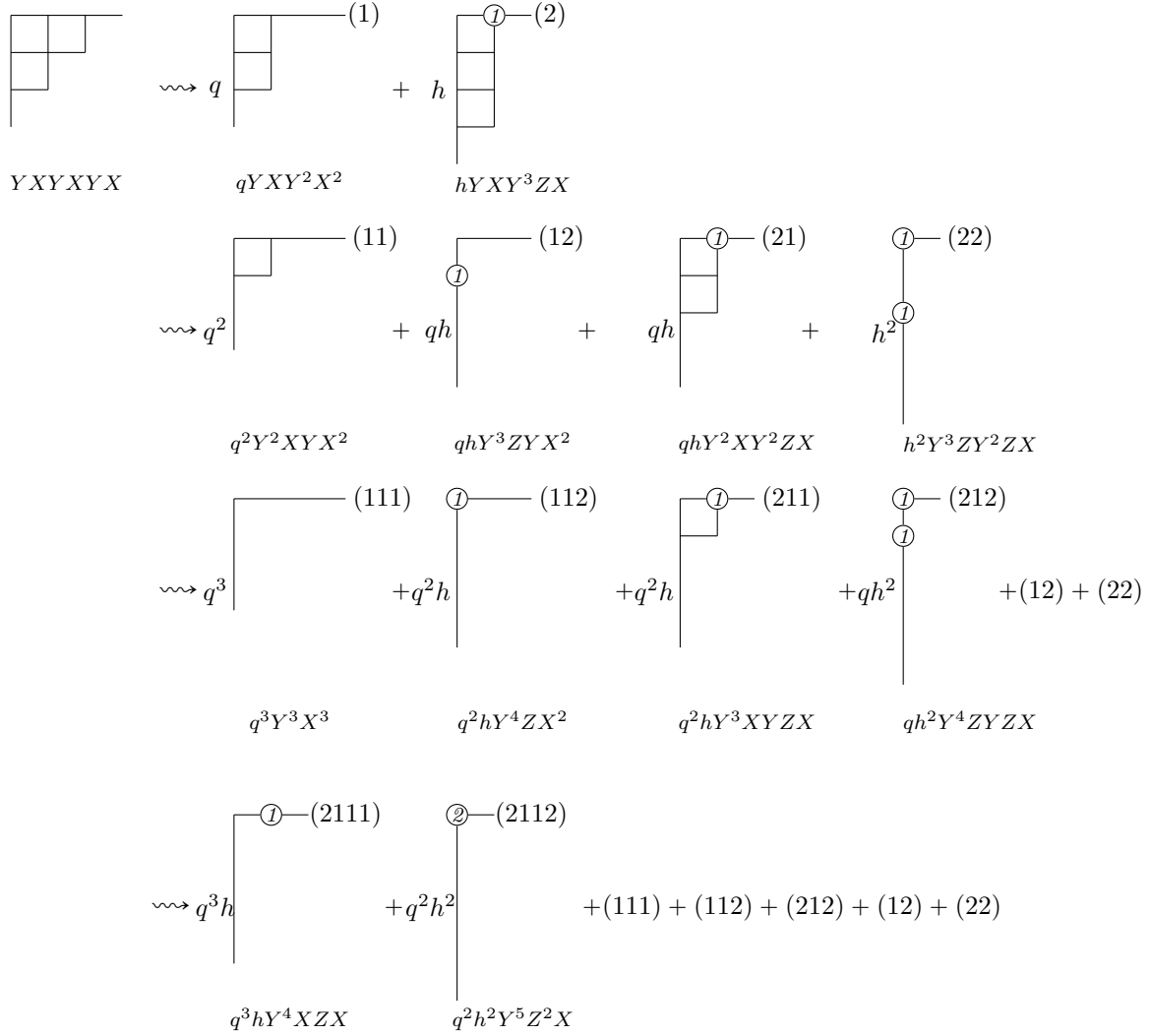
\begin{figure}[h]
\centering
\begin{tikzpicture}

\draw (0,0.5) -- (0,2);
\draw (0.5,1) -- (0.5,2);
\draw (1,1.5) -- (1,2);
\draw (0,2) -- (1.5,2);
\draw (0,1.5) -- (1,1.5);
\draw (0,1) -- (0.5,1);
\node at (0.75,-0.25) {\footnotesize $YXYXYX$};

\node at (2.25,1) {$\rightsquigarrow$};

\node at (2.75,1) {$q $};
\draw (3,0.5) -- (3,2);
\draw (3.5,1) -- (3.5,2);
\draw (3,2) -- (4.5,2);
\draw (3,1.5) -- (3.5,1.5);
\draw (3,1) -- (3.5,1);

\node at (4.75,2) {$(1)$};
\node at (3.75,-0.25) {\footnotesize $qYXY^2X^2$};

\node at (5.25,1) {$+$};

\node at (5.75,1) {$h$};
\draw (6,0) -- (6,2);
\draw (6.5,0.5) -- (6.5,1.85);
\draw (6,2) -- (6.35,2);
\draw (6.65,2) -- (7,2);
\draw (6,1.5) -- (6.5,1.5);
\draw (6,1) -- (6.5,1);
\draw (6,0.5) -- (6.5,0.5);
\node at (7.25,2) {$(2)$};

\draw (6.5,2) circle (4pt);
\node at (6.5,2) {\footnotesize 1};

\node at (6.75,-0.25) {\footnotesize $hYXY^3ZX$};


\node at (2.25,-2.25) {$\rightsquigarrow$};

\node at (2.75,-2.25) {$q^2 $};
\draw (3,-2.5) -- (3,-1);
\draw (3.5,-1.5) -- (3.5,-1);
\draw (3,-1) -- (4.5,-1);
\draw (3,-1.5) -- (3.5,-1.5);
\node at (4.9,-1) {$(11)$};

\node at (3.75,-3.75) {\footnotesize $q^2Y^2XYX^2$};

\node at (5.25,-2.25) {$+$};

\node at (5.75,-2.25) {$qh$};
\draw (6,-3) -- (6,-1.65);
\draw (6,-1.35) -- (6,-1);
\draw (6,-1) -- (7,-1);
\node at (7.4,-1) {$(12)$};

\draw (6,-1.5) circle (4pt);
\node at (6,-1.5) {\footnotesize 1};

\node at (6.75,-3.75) {\footnotesize $qhY^3ZYX^2$};

\node at (7.75,-2.25) {$+$};

\node at (8.75,-2.25) {$qh$};
\draw (9,-3) -- (9,-1);
\draw (9.5,-2) -- (9.5,-1.15);
\draw (9,-1) -- (9.35,-1);
\draw (9.65,-1) -- (10,-1);
\draw (9,-1.5) -- (9.5,-1.5);
\draw (9,-2) -- (9.5,-2);

\node at (10.4,-1) {$(21)$};
\draw (9.5,-1) circle (4pt);
\node at (9.5,-1) {\footnotesize 1};

\node at (9.9,-3.75) {\footnotesize $qhY^2XY^2ZX$};

\node at (10.75,-2.25) {$+$};

\node at (11.75,-2.25) {$h^2$};
\draw (12,-3.5) -- (12,-2.15);
\draw (12,-1.85) -- (12,-1.15);
\draw (12.15,-1) -- (12.5,-1);
\node at (12.9,-1) {$(22)$};

\draw (12,-1) circle (4pt);
\node at (12,-1) {\footnotesize 1};

\draw (12,-2) circle (4pt);
\node at (12,-2) {\footnotesize 1};

\node at (12.9,-3.75) {\footnotesize $h^2Y^3ZY^2ZX$};


\node at (2.25,-5.75) {$\rightsquigarrow$};

\node at (2.75,-5.75) {$q^3 $};
\draw (3,-6) -- (3,-4.5);
\draw (3,-4.5) -- (4.5,-4.5);
\node at (5,-4.5) {$(111)$};

\node at (3.75,-7.5) {\footnotesize $q^3Y^3X^3$};

\node at (5.25,-5.75) {$+$};

\node at (5.65,-5.75) {$q^2h$};
\draw (6,-6.5) -- (6,-4.65);
\draw (6.15,-4.5) -- (7,-4.5);
\node at (7.5,-4.5) {$(112)$};

\draw (6,-4.5) circle (4pt);
\node at (6,-4.5) {\footnotesize 1};

\node at (6.75,-7.5) {\footnotesize $q^2h Y^4ZX^2$};

\node at (8.25,-5.75) {$+$};

\node at (8.65,-5.75) {$q^2h$};
\draw (9,-6.5) -- (9,-4.5);
\draw (9.5,-5) -- (9.5,-4.65);
\draw (9,-4.5) -- (9.35,-4.5);
\draw (9.65,-4.5) -- (10,-4.5);
\draw (9,-5) -- (9.5,-5);
\node at (10.5,-4.5) {$(211)$};

\draw (9.5,-4.5) circle (4pt);
\node at (9.5,-4.5) {\footnotesize 1};

\node at (9.9,-7.5) {\footnotesize $q^2h Y^3XYZX$};

\node at (11.25,-5.75) {$+$};

\node at (11.65,-5.75) {$qh^2$};
\draw (12,-7) -- (12,-5.15);
\draw (12,-4.85) -- (12,-4.65);

\draw (12.15,-4.5) -- (12.5,-4.5);

\node at (13,-4.5) {$(212)$};

\draw (12,-4.5) circle (4pt);
\node at (12,-4.5) {\footnotesize 1};
\draw (12,-5) circle (4pt);
\node at (12,-5) {\footnotesize 1};

\node at (12.9,-7.5) {\footnotesize $qh^2 Y^4ZYZX$};

\node at (14.25,-5.75) {$+ (12) + (22)$};


\node at (2.15,-10) {$\rightsquigarrow$};

\node at (2.75,-10) {$q^3h$};
\draw (3,-10.75) -- (3,-8.75);
\draw (3,-8.75) -- (3.35,-8.75);
\draw (3.65,-8.75) -- (4,-8.75);
\node at (4.5,-8.75) {$(2111)$};

\draw (3.5,-8.75) circle (4pt);
\node at (3.5,-8.75) {\footnotesize 1};

\node at (3.75,-11.5) {\footnotesize $q^3 h Y^4XZX$};

\node at (5.15,-10) {$+$};

\node at (5.65,-10) {$q^2h^2$};
\draw (6,-11.25) -- (6,-8.9);
\draw (6.15,-8.75) -- (6.5,-8.75);
\node at (7,-8.75) {$(2112)$};

\draw (6,-8.75) circle (4pt);
\node at (6,-8.75) {\footnotesize 2};

\node at (6.75,-11.5) {\footnotesize $q^2 h^2 Y^5Z^2X$};

\node at (10.5,-10) {$+ (111) + (112) + (212) + (12) + (22)$};
\end{tikzpicture}
\caption{Normal ordering $YXYXYX$ using augmented Ferrers boards $(s=2)$.}\label{ABNOS}
\end{figure}
In Figure~\ref{ABNOS}, in the bottom row the nearly final result is shown. Commuting the letter $Z$ to its middle position will give for each of the boards $(12), (212), (2111)$ a factor of $p$, while giving a factor of $p^2$ for the board $(22)$. Thus, we find
\begin{equation}\label{ABNOEQS}
(YX)^3=q^3Y^3X^3 + \left(q^2 + pq + pq^3\right) h Y^4ZX^2  +\left(p^2 + pq + q^2\right) h^2 Y^5Z^2X.´
\end{equation}
Comparing this to \eqref{StirGen}, we see that
\begin{equation}\label{ABNOEQSStir}
\mathfrak{S}_{2;h}(3,3 |p,q)=q^3, \, \mathfrak{S}_{2;h}(3,2 |p,q)=\left(q^2 + pq + pq^3\right)h, \, \mathfrak{S}_{2;h}(3,1 |p,q)=\left(p^2 + pq + q^2\right) h^2.
\end{equation}
\end{example}

We now introduce, for $s>0$, certain $(p,q)$-weights for $s$-rook placements which generalize the one for $p=1$ introduced by Celeste et al. \cite{CCG2017} and considered in the preceding section, see \eqref{RookCeleste}. 

Let us denote again a rook placement satisfying the $s$-row creation rule by $\phi$, and the set of all rook placements of $k$ rooks satisfying the $s$-row creation rule on the Ferrers board $B$ by $\mathcal{R}_s(B, k)$. Let $\phi \in \mathcal{R}_s(B, k)$. We partition the boxes of $B$ into 5 disjoint classes:
\begin{enumerate}
\item A {\em rook box} is a box of $B$ occupied by a rook.
\item A {\em $t$-box} is a box which is not a rook box and which is lying above a rook and not lying to the left of another rook.
\item An {\em $\ell $-box} is a box which is not a rook box and which is lying to the left of a rook and which is not lying in the column of another rook.
\item A {\em cancelled box} is a box which is not a rook box and which is lying to the left of a rook and lying in the column of another rook.
\item All remaining boxes of $B$ are {\em $e$-boxes} (for ``empty'').
\end{enumerate}
\begin{definition} 
The {\em $(p,q)$-weight} of a placement $\phi$ of $k$ rooks satisfying the $s$-row creation rule on the Ferrers board $B$ is defined to be
\begin{equation}\label{PQWeightS}
\omega_{s;p,q}(B;\phi) :=h^k p^{\# t-boxes + \# \ell-boxes}q^{\# e-boxes}.
\end{equation}
\end{definition} 

\begin{example}
In Figure~\ref{FigFileSRook}, a Ferrers board $B$ is seen where a placement $\phi$ of $5$ rooks has been made with $s=2$. In total there are 84 boxes: $5$ rook boxes, $21$ $t$-boxes, $10$ $\ell$-boxes, $10$ cancelled boxes and $38$ $e$-boxes. Thus, the  $(p,q)$-weight of the placement $\phi$ is given by $h^5p^{31}q^{38}$.
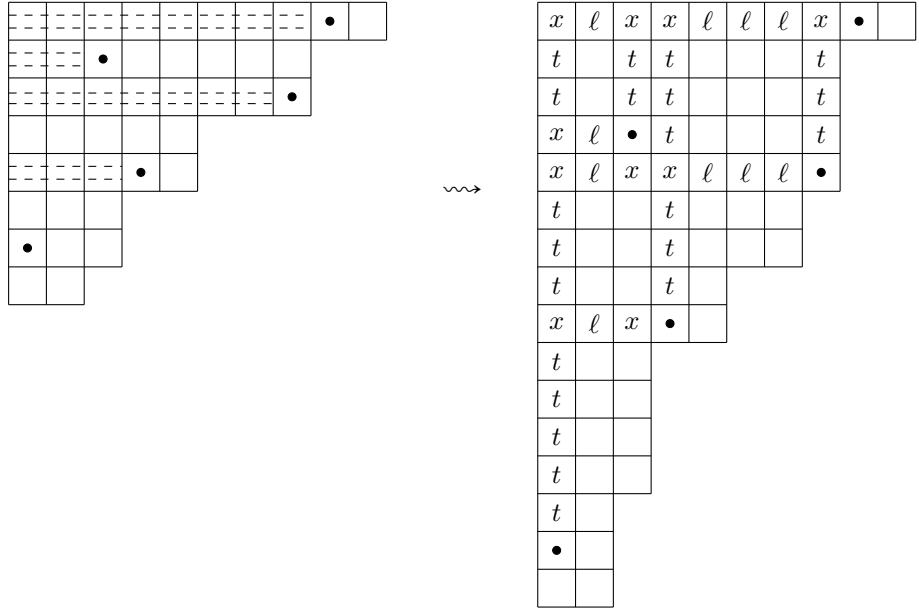
\begin{figure}[h]
\centering
\begin{tikzpicture}

\draw (0,4.5) -- (5,4.5);
\draw (0,4) -- (5,4);
\draw (0,3.5) -- (4,3.5);
\draw (0,3) -- (4,3);
\draw (0,2.5) -- (2.5,2.5);
\draw (0,2) -- (2.5,2);
\draw (0,1.5) -- (1.5,1.5);
\draw (0,1) -- (1.5,1);
\draw (0,0.5) -- (1,0.5);

\draw[dashed] (0,4.33) -- (4,4.33);
\draw[dashed] (0,4.16) -- (4,4.16);

\draw[dashed] (0,3.83) -- (1,3.83);
\draw[dashed] (0,3.66) -- (1,3.66);

\draw[dashed] (0,3.3) -- (3.5,3.3);
\draw[dashed] (0,3.16) -- (3.5,3.16);

\draw[dashed] (0,2.33) -- (1.5,2.33);
\draw[dashed] (0,2.16) -- (1.5,2.16);

\draw (0,0.5) -- (0,4.5);
\draw (0.5,0.5) -- (0.5,4.5); 
\draw (1,0.5) -- (1,4.5); 
\draw (1.5,1) -- (1.5,4.5);
\draw (2,2) -- (2,4.5);
\draw (2.5,2) -- (2.5,4.5);
\draw (3,3) -- (3,4.5);
\draw (3.5,3) -- (3.5,4.5);
\draw (4,3) -- (4,4.5);
\draw (4.5,4) -- (4.5,4.5);
\draw (5,4) -- (5,4.5);

\filldraw (4.25,4.25) circle (1.5pt);
\filldraw (1.25,3.75) circle (1.5pt);
\filldraw (3.75,3.25) circle (1.5pt);
\filldraw (1.75,2.25) circle (1.5pt);
\filldraw (0.25,1.25) circle (1.5pt);

\node at (6,2) {$\rightsquigarrow $};
\draw (7,4.5) -- (12,4.5);
\draw (7,4) -- (12,4);
\draw (7,3.5) -- (11,3.5);
\draw (7,3) -- (11,3);
\draw (7,2.5) -- (11,2.5);
\draw (7,2) -- (11,2);
\draw (7,1.5) -- (10.5,1.5);
\draw (7,1) -- (10.5,1);
\draw (7,0.5) -- (9.5,0.5);
\draw (7,0) -- (9.5,0);
\draw (7,-0.5) -- (8.5,-0.5);
\draw (7,-1) -- (8.5,-1);
\draw (7,-1.5) -- (8.5,-1.5);
\draw (7,-2) -- (8.5,-2);
\draw (7,-2.5) -- (8,-2.5);
\draw (7,-3) -- (8,-3);
\draw (7,-3.5) -- (8,-3.5);

\draw (7,-3.5) -- (7,4.5);
\draw (7.5,-3.5) -- (7.5,4.5);
\draw (8,-3.5) -- (8,4.5);
\draw (8.5,-2) -- (8.5,4.5);
\draw (9,0) -- (9,4.5);
\draw (9.5,0) -- (9.5,4.5);
\draw (10,1) -- (10,4.5);
\draw (10.5,1) -- (10.5,4.5);
\draw (11,2) -- (11,4.5);
\draw (11.5,4) -- (11.5,4.5);
\draw (12,4) -- (12,4.5);

\filldraw (11.25,4.25) circle (1.5pt);
\filldraw (10.75,2.25) circle (1.5pt);
\filldraw (8.75,0.25) circle (1.5pt);
\filldraw (8.25,2.75) circle (1.5pt);
\filldraw (7.25,-2.75) circle (1.5pt);

\node at (7.25,4.25) {$x$};
\node at (7.25,3.75) {$t$};
\node at (7.25,3.25) {$t$};
\node at (7.25,2.75) {$x$};
\node at (7.25,2.25) {$x$};
\node at (7.25,1.75) {$t$};
\node at (7.25,1.25) {$t$};
\node at (7.25,0.75) {$t$};
\node at (7.25,0.25) {$x$};
\node at (7.25,-0.25) {$t$};
\node at (7.25,-0.75) {$t$};
\node at (7.25,-1.25) {$t$};
\node at (7.25,-1.75) {$t$};
\node at (7.25,-2.25) {$t$};

\node at (7.75,4.25) {$\ell$};
\node at (7.75,2.75) {$\ell$};
\node at (7.75,2.25) {$\ell$};
\node at (7.75,0.25) {$\ell$};

\node at (8.25,4.25) {$x$};
\node at (8.25,3.75) {$t$};
\node at (8.25,3.25) {$t$};
\node at (8.25,2.25) {$x$};
\node at (8.25,0.25) {$x$};

\node at (8.75,4.25) {$x$};
\node at (8.75,3.75) {$t$};
\node at (8.75,3.25) {$t$};
\node at (8.75,2.75) {$t$};
\node at (8.75,2.25) {$x$};
\node at (8.75,1.75) {$t$};
\node at (8.75,1.25) {$t$};
\node at (8.75,0.75) {$t$};

\node at (9.25,4.25) {$\ell$};
\node at (9.25,2.25) {$\ell$};

\node at (9.75,4.25) {$\ell$};
\node at (9.75,2.25) {$\ell$};

\node at (10.25,4.25) {$\ell$};
\node at (10.25,2.25) {$\ell$};

\node at (10.75,4.25) {$x$};
\node at (10.75,3.75) {$t$};
\node at (10.75,3.25) {$t$};
\node at (10.75,2.75) {$t$};

\end{tikzpicture}
\caption{The Ferrers board $B_{8875533311}$ with a $5$-rook placement for $s=2$ (left) and the equivalent representation with boxes marked by their class (right).}\label{FigFileSRook}
\end{figure} 
\end{example}

In analogy to \eqref{RookCeleste} (for $s>0$ but $p=1$) and \eqref{PQRookWe} (for $s=0$ but arbitrary $p$) we define the associated {\em $(p,q)$-deformed $s$-rook numbers} as follows.
\begin{definition}
Let $s>0$. Let $B$ be a Ferrers board and let $\mathcal{R}_s(B, k)$ be the set of placements of $k$ rooks on $B$ satisfying the $s$-row creation rule. Then the {\em $(p,q)$-deformed $s$-rook numbers} are defined by
\begin{equation}\label{PQRookWeS}
\mathscr{R}_{s,h; p,q}[B, k] :=\sum\limits_{\phi\in\mathcal{R}_s(B, k)} \omega_{s;p,q}(B;\phi)=h^k \sum\limits_{\phi\in\mathcal{R}_s(B, k)} p^{\# t-boxes + \# \ell-boxes}q^{\# e-boxes}.
\end{equation}
\end{definition}
Clearly, $\mathscr{R}_{s,h; p,q}[B, k] =h^k \mathscr{R}_{s,1; p,q}[B, k] $. Observe that for $p=1$ the weight of a placement reduces to $\omega_{s;p,q}(B;\phi) =h^k q^{ \# e-boxes}$, i.e., the parameter $q$ counts the number of boxes not containing a rook and not lying above or to the left of a rook. But this is exactly the exponent denoted by $v_s(\phi)$ in Section~\ref{SectNORook}. (Recall that we use here the original interpretation of $s$-rook placements where $s$ additional rows are generated when placing a rook and where the rooks are non-attacking, while above the rows to the left are divided into $s$ rows and there is no such restriction on placing the subsequent rooks, see Remark~\ref{IRookEq}.)

Thus, for $p=1$, the $(p,q)$-deformed $s$-rook numbers introduced above reduce to those given in \eqref{RookCeleste}, i.e.,
$$
\mathscr{R}_{s,h; 1,q}[B, k] =R_{s, h; q}[B, k].
$$
For $s=0$, we identified in \eqref{PQStirlingRookEq} the $(p,q)$-deformed Stirling numbers of the second kind as $(p,q)$-deformed rook numbers of the staircase board, thereby generalizing the first equation of \eqref{StairEq}. The unsigned Stirling numbers $|s(n,k)|$ -- also called {\em cycle numbers} $c(n,k)=|s(n,k)|$ -- have a similar interpretation as file numbers, see second equation of \eqref{StairEq}, or, using \eqref{RookRed}, as $1$-rook numbers, 
$$
c(n,k)=f_{n-k}(J_n)=r_{n-k}^{(1)}(J_n).
$$
\begin{definition}
Let $n,k \in \mathbb{N}_0$ and let $k\leq n$. The {\em $(p,q)$-deformed cycle numbers} are defined by
\begin{equation}\label{pqcycleWe}
\hat{c}_{p,q}(n,k):=\mathscr{R}_{1,1; p,q}[J_n, n-k].
\end{equation}
\end{definition}
Clearly, for $p=q=1$, they reduce to the conventional cycle numbers, $\hat{c}_{1,1}(n,k)=c(n,k)$. Using the combinatorial interpretation, we can derive the recurrence relation for $\hat{c}_{p,q}(n,k)$ in the same fashion as for the $(p,q)$-deformed Stirling  numbers of the second kind, see Proposition~\ref{PQPropStirRook}.

\begin{proposition}\label{PQPropCycRook} The $(p,q)$-deformed cycle numbers $\hat{c}_{p,q}(n,k)$ satisfy $\hat{c}_{p,q}(0,0)=1$ and the recurrence relation 
$$
\hat{c}_{p,q}(n+1,k)=p^{n-k+1}q^{n}\hat{c}_{p,q}(n,k-1)+[n]_{p,q}\hat{c}_{p,q}(n,k).
$$
\end{proposition}
\begin{proof}
The proof is very similar to the one given in Proposition~\ref{PQPropStirRook}. Let us consider $\hat{c}_{p,q}(n+1,k)= \mathscr{R}_{1,1; p,q}[J_{n+1},n-k+1]$, i.e., the set of $(n-k+1)$ file placements on $J_{n+1}$. 

Let $\phi$ be a file placement of Type I. Considered as a placement on $J_{n}$ it has weight $\omega_{1;p,q}(J_n;\phi)$. However, when considered as a placement on $J_{n+1}$ it has weight $\omega_{1;p,q}(J_{n+1};\phi)=p^{n-k+1}q^{n}\omega_{1;p,q}(J_n;\phi)$ since in the adjoined column with $n$ boxes one has no $t$-boxes, $(n-k+1)$ $\ell $-boxes (the rows of the rooks) and $n$ $e$-boxes (the rows without a rook). Since this holds for all placements of type I, their contribution is given by
$$
p^{n-k+1}q^{n}\sum\limits_{\phi\in\mathcal{R}_1(J_n, n-k+1)} \omega_{1;p,q}(J_n;\phi)=p^{n-k+1}q^{n}\mathscr{R}_{1,1; p,q}[J_n,n-k+1]=p^{n-k+1}q^{n}\hat{c}_{p,q}(n,k-1).
$$
Now, let us turn to the placements of type II. Let $\phi'$ be a file placement of $(n-k)$ rooks on $J_n$ with weight $\omega_{1;p,q}(J_n;\phi')$. For each such placement we obtain $n$ placements on $J_{n+1}$ by choosing for the $(n-k+1)$-th rook one of the $n$ boxes. Placing this rook into the $r$-th box from above (where $r=1,\ldots,n$) will result in $(r-1)$ additional $t$-boxes and $(k-r+1)$ additional $e$-boxes, implying for this placement the weight $p^{r-1}q^{k-r+1}\omega_{1;p,q}(J_n;\phi')$. Thus, summing the contribution of the weights of all $n$ placements on $J_{n+1}$ resulting from the placement $\phi'$, one finds
$$
(p^0q^{n-1}+p^1q^{n-2}+\cdots+p^{n-2}q^{1}+p^{n-1}q^{0})\, \omega_{1;p,q}(J_n;\phi')=[n]_{p,q}\, \omega_{1;p,q}(J_n;\phi').
$$
Since this holds for all placements of type II, their contribution is given by
$$
[n]_{p,q}\sum\limits_{\phi'\in\mathcal{R}_1(J_n, n-k)} \omega_{1;p,q}(J_n;\phi')=[n]_{p,q}\mathscr{R}_{1,1; p,q}[J_n,n- k]=[n]_{p,q}\hat{c}_{p,q}(n,k).
$$
Adding the contributions from placements of type I and type II yields the assertion. 
\end{proof}
Comparing the recurrence relation for $\hat{c}_{p,q}(n,k)$ given above with \eqref{Stirlingrecu}, for $s=1$, we have the identification $\hat{c}_{p,q}(n,k)=\mathfrak{S}_{1;1}(n,k|p,q)$.  

\begin{remark}
A standard way to define the $(p,q)$-deformed cycle numbers is to define them in terms of the elementary symmetric functions, $c_{p,q}(n,k)=e_{n-k}([1]_{p,q},[2]_{p,q},\ldots,[n-1]_{p,q})$. They satisfy the recurrence $c_{p,q}(n+1,k)=c_{p,q}(n,k-1)+[n]_{p,q}c_{p,q}(n,k)$, see \cite{dML1993}. Thus, the $(p,q)$-deformed cycle numbers $\hat{c}_{p,q}(n,k)$ defined in \eqref{pqcycleWe} do not coincide with the standard definition $c_{p,q}(n,k)$. For $p=1$, both versions reduce to the same $q$-deformed cycle numbers.
\end{remark}

In general, we have the following result.
\begin{theorem}\label{TheoWordNOS}
Let $\mathrm{w}_{{\bf m}, {\bf n}}=Y^{n_r}X^{m_r}\cdots Y^{n_1}X^{m_1}$ be a word in the letters $X$ and $Y$ satisfying the commutation relations \eqref{pqdefgenWeyl} of the $(p,q)$-deformed generalized Weyl algebra $A_{s;h|p,q}$. Then one has the normal ordering result
\begin{equation}\label{pqrookNOS}
\mathrm{w}_{{\bf m}, {\bf n}}=\sum_{k=0}^{|{\bf m}|-m_1} \mathscr{R}_{s,h; p,q}[B_{\mathrm{w}_{{\bf m}, {\bf n}}}, k] \, Y^{|{\bf n}|+(s-1)k}Z_p^{k}X^{|{\bf m}|-k},
\end{equation}
where $B_{\mathrm{w}_{{\bf m}, {\bf n}}}$ is the Ferrers board associated to $\mathrm{w}_{{\bf m}, {\bf n}}$ and $\mathscr{R}_{s,h; p,q}[B_{\mathrm{w}_{{\bf m}, {\bf n}}}, k]$ is given by \eqref{PQRookWeS}.
\end{theorem}
Letting $\ell = |{\bf m}|-k$ in \eqref{pqrookNOS}, a comparison with \eqref{pqrook} shows that one has in analogy to \eqref{Identification} the identification
\begin{equation}\label{IdentStirRookS}
S_{s, h; p, q}^{{\bf m}, {\bf n}}[\ell]=\mathscr{R}_{s,h; p,q}[B_{\mathrm{w}_{{\bf m}, {\bf n}}}, |{\bf m}|-\ell] .
\end{equation}

\begin{corollary}\label{PQStirlingRookS}
Let us consider the word $\omega=(YX)^n$, for $n \in \mathbb{N}$. Since $B_{(YX)^n}=J_{n}$, \eqref{pqrookNOS} implies that
\begin{equation}\label{PQStirlingRookSEQ}
(YX)^n=\sum_{\ell=0}^{n} \mathscr{R}_{s,h; p,q}[J_{n}, n-\ell ] \, Y^{s(n-\ell)+\ell }Z_p^{n-\ell}X^{\ell}.
\end{equation}
Comparing with \eqref{StirGen}, we find $\mathfrak{S}_{s;h}(n,\ell |p,q)=\mathscr{R}_{s,h; p,q}[J_{n}, n-\ell ]$, for $n\geq 1$. 
\end{corollary}
\begin{example}
Let us choose $s=2$ and consider $(YX)^3$. From \eqref{PQStirlingRookSEQ}, we have 
$$
(YX)^3= \mathscr{R}_{2,h; p,q}[J_{3}, 2] \, Y^5Z_p^2X+\mathscr{R}_{2,h; p,q}[J_{3}, 1] \, Y^4 Z_pX^2 + \mathscr{R}_{2,h; p,q}[J_{3}, 0] \, Y^3X^3.
$$
Using $\mathfrak{S}_{2;h}(n,\ell |p,q) =\mathscr{R}_{2,h; p,q}[J_{n}, n-\ell ]$ and \eqref{ABNOEQSStir}, we find that
$$
\mathscr{R}_{2,h; p,q}[J_{3}, 0]=q^3, \,\, \mathscr{R}_{2,h; p,q}[J_{3}, 1]=\left(q^2 + pq + pq^3\right)h, \,\, \mathscr{R}_{2,h; p,q}[J_{3}, 2]=\left(p^2 + pq + q^2\right) h^2.
$$
For comparison, we have drawn the corresponding rook placements and their $(p,q)$-weights in Figure~\ref{ABNOS2}, showing, e.g., that $\mathscr{R}_{2,h; p,q}[J_{3}, 1]=\left(q^2 + pq + pq^3\right)h$. The label of the corresponding normal ordered augmented Ferrers board from Figure~\ref{ABNOS} is also given. 
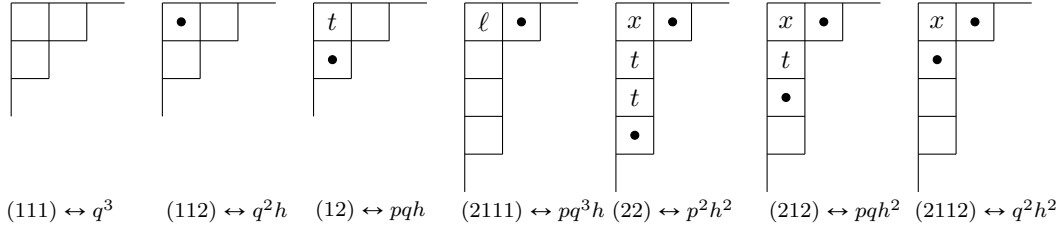
\begin{figure}[h]
\centering
\begin{tikzpicture}

\draw (0,0.5) -- (0,2);
\draw (0.5,1) -- (0.5,2);
\draw (1,1.5) -- (1,2);
\draw (0,2) -- (1.5,2);
\draw (0,1.5) -- (1,1.5);
\draw (0,1) -- (0.5,1);
\node at (0.65,-0.75) {\footnotesize $(111)\leftrightarrow q^3$};

\draw (2,0.5) -- (2,2);
\draw (2.5,1) -- (2.5,2);
\draw (3,1.5) -- (3,2);
\draw (2,2) -- (3.5,2);
\draw (2,1.5) -- (3,1.5);
\draw (2,1) -- (2.5,1);
\filldraw (2.25,1.75) circle (1.5pt);
\node at (2.85,-0.75) {\footnotesize $(112)\leftrightarrow q^2h $};

\draw (4,0.5) -- (4,2);
\draw (4.5,1) -- (4.5,2);
\draw (5,1.5) -- (5,2);
\draw (4,2) -- (5.5,2);
\draw (4,1.5) -- (5,1.5);
\draw (4,1) -- (4.5,1);
\node at (4.25,1.75) {$t$};
\filldraw (4.25,1.25) circle (1.5pt);
\node at (4.75,-0.75) {\footnotesize $(12)\leftrightarrow pqh$};

\draw (6,-0.5) -- (6,2);
\draw (6.5,0) -- (6.5,2);
\draw (7,1.5) -- (7,2);
\draw (6,2) -- (7.5,2);
\draw (6,1.5) -- (7,1.5);
\draw (6,1) -- (6.5,1);
\draw (6,0.5) -- (6.5,0.5);
\draw (6,0) -- (6.5,0);
\node at (6.25,1.75) {$\ell$};
\filldraw (6.75,1.75) circle (1.5pt);
\node at (6.9,-0.75) {\footnotesize $(2111)\leftrightarrow pq^3h$};

\draw (8,-0.5) -- (8,2);
\draw (8.5,0) -- (8.5,2);
\draw (9,1.5) -- (9,2);
\draw (8,2) -- (9.5,2);
\draw (8,1.5) -- (9,1.5);
\draw (8,1) -- (8.5,1);
\draw (8,0.5) -- (8.5,0.5);
\draw (8,0) -- (8.5,0);
\node at (8.25,1.75) {$x$};
\node at (8.25,1.25) {$t$};
\node at (8.25,0.75) {$t$};
\filldraw (8.25,0.25) circle (1.5pt);
\filldraw (8.75,1.75) circle (1.5pt);
\node at (8.75,-0.75) {\footnotesize $(22)\leftrightarrow p^2h^2$};

\draw (10,-0.5) -- (10,2);
\draw (10.5,0) -- (10.5,2);
\draw (11,1.5) -- (11,2);
\draw (10,2) -- (11.5,2);
\draw (10,1.5) -- (11,1.5);
\draw (10,1) -- (10.5,1);
\draw (10,0.5) -- (10.5,0.5);
\draw (10,0) -- (10.5,0);
\node at (10.25,1.75) {$x$};
\node at (10.25,1.25) {$t$};
\filldraw (10.25,0.75) circle (1.5pt);
\filldraw (10.75,1.75) circle (1.5pt);
\node at (10.9,-0.75) {\footnotesize $(212)\leftrightarrow pqh^2$};

\draw (12,-0.5) -- (12,2);
\draw (12.5,0) -- (12.5,2);
\draw (13,1.5) -- (13,2);
\draw (12,2) -- (13.5,2);
\draw (12,1.5) -- (13,1.5);
\draw (12,1) -- (12.5,1);
\draw (12,0.5) -- (12.5,0.5);
\draw (12,0) -- (12.5,0);
\node at (12.25,1.75) {$x$};
\filldraw (12.25,1.25) circle (1.5pt);
\filldraw (12.75,1.75) circle (1.5pt);
\node at (12.9,-0.75) {\footnotesize $(2112)\leftrightarrow q^2h^2$};

\end{tikzpicture}
\caption{$(p,q)$-weights of rook placements on the Ferrers board of $(YX)^3$ ($s=2$).}\label{ABNOS2}
\end{figure}

\end{example}

Similar to the case $s=0$, we can also consider normal ordering words containing letters $Z_p$. In fact, the considerations which led to Proposition~\ref{PropPQWeylNO} were independent from $s$, so they can be used here verbatim, leading to the following result.
\begin{proposition}\label{PropPQWeylNOS}
Let $\mathrm{w}_{{\bf m},{\bf n},{\bf u}}$ be an arbitrary word in the letters $X,Y$ and $Z_p$ satisfying the commutation relations \eqref{pqdefgenWeyl} of the $(p,q)$-deformed generalized Weyl algebra $A_{s;h|p,q}$. Then one has the normal ordering result
\begin{equation}\label{pqrookNOGS}
\mathrm{w}_{{\bf m},{\bf n},{\bf u}}=\sum_{k=0}^{|{\bf m}|-m_1} p^{|{\bf u}|(|{\bf m}|-k)+\mathscr{L}({\bf m},{\bf n},{\bf u})} \mathscr{R}_{s,h; p,q}[B_{\mathrm{w}_{{\bf m},{\bf n},{\bf u}}^{\ast}}, k] \, Y^{|{\bf n}|+(s-1)k}Z_p^{|{\bf u}|+k}X^{|{\bf m}|-k},
\end{equation}
where $\mathrm{w}_{{\bf m},{\bf n},{\bf u}}^{\ast}$ denotes the reduced word of $\mathrm{w}_{{\bf m},{\bf n},{\bf u}}$ and $\mathscr{L}({\bf m},{\bf n},{\bf u})$ is defined in \eqref{PQExponent}.
\end{proposition}
By comparing coefficients, we identified $\mathfrak{S}_{s;h}(n,\ell |p,q)=\mathscr{R}_{s,h; p,q}[J_{n}, n-\ell ]$ in Corollary~\ref{PQStirlingRookS}. Similar to the case $s=0$ (see, in particular, Proposition~\ref{PQPropStirRook}), we want to give a combinatorial interpretation for the recurrence relation of the generalized Stirling numbers $\mathfrak{S}_{s;h}(n,\ell |p,q)$. In analogy to above, we define  
$$
T_{s,h;p,q}(n,k):= \mathscr{R}_{s,h; p,q}[J_n,n- k].
$$
\begin{proposition}\label{PQPropStirRookS} The numbers $T_{s,h;p,q}(n,k)$ satisfy $T_{s,h;p,q}(0,0)=1$ and the recurrence relation 
$$
T_{s,h;p,q}(n+1,k)=p^{n-k+1}q^{s(n-k+1)+k-1}T_{s,h;p,q}(n,k-1)+h[s(n-k)+k]_{p,q}T_{s,h;p,q}(n,k).
$$
\end{proposition}
\begin{proof}
The argument is almost the same as for the case $s=0$. Let us consider $T_{s,h;p,q}(n+1,k)= \mathscr{R}_{s,h; p,q}[J_{n+1},n- k+1]$, i.e., the set of $(n-k+1)$ rook placements on $J_{n+1}$. 

We start with the rook placements of Type I. Let $\phi$ be such a placement. Considered as a placement on $J_{n}$ it has weight $\omega_{p,q}(J_n;\phi)$. Due to the $s$-row creation rule, placing $(n-k+1)$ rooks leads to $s(n-k+1)$ additional boxes in the adjoined column to the left, which has in total $n+s(n-k+1)$ boxes. In this column there exist $(n-k+1)$ $\ell$-boxes, hence $n+(s-1)(n-k+1)$ $e$-boxes. Thus, considering $\phi$ as a placement on $J_{n+1}$ it has weight $\omega_{p,q}(J_{n+1};\phi)=p^{n-k+1}q^{s(n-k+1)+k-1}\omega_{p,q}(J_n;\phi)$. Since this holds for all placements of type I, their contribution is given by
$$
p^{n-k+1}q^{s(n-k+1)+k-1}\sum\limits_{\phi\in\mathcal{R}_s(J_n, n-k+1)} \omega_{p,q}(J_n;\phi)=p^{n-k+1}q^{s(n-k+1)+k-1}\mathscr{R}_{s,h; p,q}[J_n,n- k+1],
$$
and this equals $p^{n-k+1}q^{s(n-k+1)+k-1}T_{s,h;p,q}(n,k-1)$.

Now, let us turn to the placements of type II. Let $\phi'$ be a placement of $(n-k)$ rooks on $J_n$ with weight $\omega_{p,q}(J_n;\phi')$. Due to the $s$-row creation rule, placing $(n-k)$ rooks leads to additional $s(n-k)$ boxes in the adjoined column to the left which has in total $n+s(n-k)$ boxes. We have $n+(s-1)(n-k)=s(n-k)+k$ ``free'' boxes where we can place a rook (we are not allowed to place it in a row where one of the $(n-k)$ rooks lies). Placing the rook in the $r$-th box from above (where $r=1,\ldots, s(n-k)+k$) will result in one additional rook box, $(r-1)$ additional $t$-boxes, and $s(n-k)+k-r$ additional $e$-boxes (and $(n-k)$ cancelled boxes which do not contribute to the weight), implying for this placement the weight $h p^{r-1}q^{s(n-k)+k-r}\omega_{p,q}(J_n;\phi')$. Thus, summing the contribution of the weights of all $s(n-k)+k$ placements on $J_{n+1}$ resulting from the placement $\phi'$, one finds
$$
h(p^0q^{s(n-k)+k-1}+\cdots+p^{s(n-k)+k-1}q^{0})\,\omega_{p,q}(J_n;\phi')=h[s(n-k)+k]_{p,q}\,\omega_{p,q}(J_n;\phi').
$$
Since this holds for all placements of type II, their contribution is given by
$$
h[s(n-k)+k]_{p,q}\sum\limits_{\phi'\in\mathcal{R}_s(J_n, n-k)} \omega_{p,q}(J_n;\phi')=h[s(n-k)+k]_{p,q}\mathscr{R}_{s,h; p,q}[J_n,n- k].
$$
The right-hand side equals $h[s(n-k)+k]_{p,q}T_{s,h;p,q}(n,k)$. Adding the contributions from placements of type I and type II yields the assertion. 
\end{proof}
The preceding proposition gives a combinatorial interpretation for the factors in the recurrence relation of the generalized Stirling numbers $\mathfrak{S}_{s;h}(n,k|p,q)$. Note that the generalized $(p,q)$-deformed Stirling numbers considered by Remmel and Wachs \cite{JRMW2004} satisfy different recurrence relations.

\begin{remark}
Recall that in the case $p=q=1$  the (unsigned) Lah numbers $L(n,k)$ are given as generalized Stirling numbers for $s=1/2$ and $h=2$. In the $q$-deformed situation, one also has a close relation, see, e.g., \cite[Section 9.4.4]{TMMS2016}. However, giving a combinatorial interpretation using rook placements on the staircase board seems to be difficult. Following Garsia and Remmel \cite{GR1980}, a rectangular board will be called {\em Laguerre board} and the board with $n$ columns of height $n-1$ will be denoted by ${\mathcal L}_n$. Since we use the same $q$-statistic as Garsia and Remmel (see Section~\ref{SectNORook}), we can state their result as $L_q(n,k)=R_{0,1;q}[{\mathcal L}_n,n-k]$. Thus, in this framework it would be natural to introduce a $(p,q)$-deformed variant by letting $L_{p,q}(n,k):=\mathscr{R}_{0,1; p,q}[{\mathcal L}_n,n- k]$ and deriving its properties and comparing with other definitions, see, e.g., \cite{RS2016,SY2017}. Schlosser and Yoo considered in \cite{SY2017} file numbers on the {\em Abel board} ${\mathcal A}_n$ consisting of $n-1$ columns of height $n$ and considered the {\em Abel numbers} as file numbers, $t_{n,k}=f_{n-k}({\mathcal A}_n)$. Thus, one has $t_{n,k}=\mathscr{R}_{1,1; 1,1}[{\mathcal A}_n,n- k]$ and it would also be interesting to define $(p,q)$-deformed versions $t_{n,k}(p,q):=\mathscr{R}_{1,1; p,q}[{\mathcal A}_n,n- k]$ and compare to the results of \cite{SY2017}. 
\end{remark}

Using the same arguments as above for deriving the recurrence relation of the generalized Stirling numbers, we can give a $(p,q)$-analog to the recurrence relation for the $q$-deformed $s$-rook numbers given by Celeste et al. \cite[Theorem 7]{CCG2017}. For this, we consider the Ferrers board $B_{\lambda_{n-1}\lambda_{n-2}\cdots \lambda_{1}\lambda_{0}}$, i.e., we label the columns from right to left, and the first column from the right has $\lambda_{0}$ boxes, while the first column to the left has $\lambda_{n-1}$ boxes. If we let $\lambda=\lambda_{n-1}\lambda_{n-2}\cdots \lambda_{1}\lambda_{0}$ be the associated partition, we denote the above board also as $B_n(\lambda)$, and $B_{n-1}(\lambda)=B_{\lambda_{n-2}\cdots  \lambda_{0}}$. 
\begin{theorem}
Let $s >0$ and let $0\leq k \leq n$. Then the $(p,q)$-deformed $s$-rook numbers $\mathscr{R}_{s,h; p,q}[B_n(\lambda),k]$ satisfy the following recursive formula 
\begin{eqnarray*}
\mathscr{R}_{s,h; p,q}[B_n(\lambda),k]&=& p^kq^{\lambda_{n-1}+k(s-1)}\mathscr{R}_{s,h; p,q}[B_{n-1}(\lambda),k] \\ && +h[\lambda_{n-1}+(k-1)(s-1)]_{p, q}\mathscr{R}_{s,h; p,q}[B_{n-1}(\lambda),k-1],
\end{eqnarray*}
with the boundary condition
$$
\mathscr{R}_{s,h; p,q}[B_n(\lambda),n] =h^n \prod_{j=0}^{n-1}[\lambda_j+(j-1)(s-1)]_{p, q},
$$
as well as $\mathscr{R}_{s,h; p,q}[B_0(\lambda),0] =1$ and $ \mathscr{R}_{s,h; p,q}[B_n(\lambda),0]=q^{|\lambda|}$, where $|\lambda|=\lambda_{n-1}+\cdots+\lambda_0$ denotes the length of the partition $\lambda$.
\end{theorem}
\begin{proof} Exactly as in the proof of Proposition~\ref{PQPropStirRookS}, we partition the set of placement of $k$ rooks on $B_n(\lambda)$ into those of type I (all $k$ rooks are placed in $B_{n-1}(\lambda)$ and the column to the left is empty) and those of type II ($(k-1)$ rooks are placed in $B_{n-1}(\lambda)$ and one rook is placed in the column to the left). The contribution of all placements of type I is given by $p^kq^{\lambda_{n-1}+k(s-1)}\mathscr{R}_{s,h; p,q}[B_{n-1}(\lambda),k] $, while the contribution of all placements of type II is given by $h[\lambda_{n-1}+(k-1)(s-1)]_{p, q}\mathscr{R}_{s,h; p,q}[B_{n-1}(\lambda),k-1]$. Adding both contributions gives the asserted recurrence. The boundary condition for $\mathscr{R}_{s,h; p,q}[B_n(\lambda),n]$ can be checked directly and follows also by iterating $n$ times placements of type II. The boundary condition for $\mathscr{R}_{s,h; p,q}[B_0(\lambda),0]$ is trivial, and the one for $ \mathscr{R}_{s,h; p,q}[B_n(\lambda),0]$ follows directly from the definition of the $(p,q)$-deformed weights since all $|\lambda|$ boxes are empty boxes when no rook is placed.
\end{proof}
\begin{remark}
For $p=1$, we recover the recurrence given in \cite[Theorem 7]{CCG2017} (note that there seem to be two typos in the boundary conditions mentioned in \cite{CCG2017}). Specializing further $q=1$, one recovers the recurrence derived by Goldman and Haglund \cite{JGJH2000}.
\end{remark}

Recalling \eqref{IdentStirRookS}, we let $S_{s, h; p, q}^{\lambda}(n, k) :=\mathscr{R}_{s,h; p,q}[B_n(\lambda),n-k]$. The preceding theorem implies that
$$
S_{s, h; p, q}^{\lambda}(n, k)=p^{n-k}q^{\lambda_{n-1}+(n-k)(s-1)}S_{s, h; p, q}^{\lambda}(n-1, k-1)+h[\lambda_{n-1}+(n-k-1)(s-1)]_{p, q}S_{s, h; p, q}^{\lambda}(n-1, k).
$$

\begin{example}
Consider the staircase board $J_n$ where $\lambda_{\ell}=\ell$ for $0\leq \ell \leq n-1$. Then the above recurrence reduces to \eqref{Stirlingrecu}. Letting $s=0$ and $h=1$, we obtain the recurrence \eqref{pqStirlingRecu} for the $(p, q)$-generalized Stirling numbers of the second kind $S_{p, q}(n, k)$ discussed in \cite{LO2020}. Moreover, for $p=1$, we obtain the recursive formula for the $q$-Stirling numbers of the second kind $S_{q}(n, k)$ (see \cite{TMMSMS2012}). On the other hand, for $p=q=1$, the above recurrence reduces to the one for the generalized Stirling numbers $\mathfrak{S}_{s; h}(n, k)$ \cite{TMMSMS2011}.
\end{example}

\section{Conclusion}\label{concl}
In this paper, we considered normal ordering in the $(p, q)$-deformed generalized Weyl algebra $A_{s;h|p,q}$ generated by variables $X, Y$, and $Z_p$, satisfying the $(p, q)$-commutation relations given by $XY-qYX=h Y^sZ_{p}, XZ_p=pZ_pX$, and $Z_pY=pYZ_p$, where $s\in \mathbb{N}_0$. Along the way to define $(p,q)$-weights for rook placements adapted to the normal ordering of variables satisfying the above commutation relations, we introduced augmented Ferrers boards as representations of words in $X,Y,$ and $Z_p$. This is a slight extension of the conventional Ferrers boards (representing words in $X,Y$) to include the additional variable $Z_p$. As noted, $Z_p$ ``almost'' commutes with $X,Y$ so this extension consists mainly in an appropriate marking of the board. By translating these properties into corresponding $(p,q)$-weights for rook placements, we defined $(p,q)$-deformed $s$-rook numbers which are adapted to the normal ordering problem for variables satisfying the above commutation relations and which reduce to the one considered by Celeste et al. \cite{CCG2017} for $p=1$. That is, the normal ordering coefficients of a word in $X,Y$ and $Z_p$ are given by the $(p,q)$-deformed $s$-rook numbers of the Ferrers board associated with the word. This interpretation allowed us, in particular, to provide a combinatorial interpretation for the recurrence relation of the $(p,q)$-deformed generalized Stirling numbers. Looking forward, this interplay opens up several exciting avenues for future research. \\
In a sequel \cite{TLM3}, we will extend this work to investigate the binomial formula within this $(p, q)$-deformed generalized Weyl algebra.

\end{document}